\documentclass[11pt,reqno]{amsart}
\usepackage[T1]{fontenc}
\usepackage[utf8]{inputenc}
\usepackage{lmodern}
\usepackage{amsmath,amsfonts,amssymb,mathtools,booktabs,mathrsfs}
\usepackage[a4paper,margin=28mm]{geometry}
\usepackage{microtype}
\usepackage{graphicx,algorithm,algorithmic,placeins,needspace}
\usepackage[hidelinks,pdfencoding=auto]{hyperref}
\numberwithin{equation}{section}
\numberwithin{figure}{section}
\numberwithin{table}{section}
\numberwithin{algorithm}{section}
\newtheorem{theorem}{Theorem}[section]
\newtheorem{lemma}[theorem]{Lemma}
\newtheorem{proposition}[theorem]{Proposition}
\newtheorem{corollary}[theorem]{Corollary}
\theoremstyle{definition}
\newtheorem{definition}[theorem]{Definition}
\theoremstyle{remark}
\newtheorem{remark}[theorem]{Remark}

\newcommand{\R}{\mathbb R}

\newcommand{\E}{\mathbb E}
\renewcommand{\P}{\mathbb P}
\newcommand{\F}{\mathcal F}

\newcommand{\dd}{\,\mathrm d}
\newcommand{\as}{\text{a.s.}}
\newcommand{\norm}[1]{\lVert#1\rVert}

\newcommand{\ind}[1]{\mathbf 1_{\{#1\}}}
\DeclareMathOperator{\sgn}{sgn}

\DeclareMathOperator{\Var}{Var}

\newcommand{\PP}{\mathcal P}

\allowdisplaybreaks[1]
\title[Stochastic Stabilization with an Unknown Input Gain]{Critical Growth for Stochastic Adaptive Stabilization with an Unknown Input Gain}
\author{Zhaobo Liu}
\address{Institute for Advanced Study, Shenzhen University, Shenzhen, Guangdong 518060, China}
\email{liuzhaobo@szu.edu.cn}
\thanks{This work was supported in part by the National Natural Science Foundation of China under Grant 12401585, the Guangdong Basic and Applied Basic Research Foundation under Grant 2024A1515011542, and the General Program of the Shenzhen Natural Science Foundation under Grant JCYJ20250604181037012.}
\hypersetup{pdftitle={Critical Growth for Stochastic Adaptive Stabilization with an Unknown Input Gain},pdfauthor={Zhaobo Liu},pdfsubject={Stochastic nonlinear adaptive control; research manuscript}}
\date{}
\begin{document}
\begin{abstract}
This paper studies adaptive stabilization of a class of scalar
discrete-time nonlinear systems with unknown drift coefficients and an
unknown nonzero input gain. Under iid Gaussian noise, we establish a
necessary and sufficient polynomial criterion for global stabilizability,
characterizing the admissible combinations of growth exponents.
When the criterion holds, we construct a feedback combining finite
calibration with adaptive control. The time average of the squared output
converges almost surely to the noise variance, attaining the smallest
possible value.
If the criterion fails, then for every Borel causal feedback there exist
fixed parameter values and an initial state for which the output magnitude
diverges with positive probability.
\end{abstract}
\maketitle

\section{Introduction}\label{sec:intro}
Designing stabilizing feedback with incomplete model information is a
longstanding problem in control theory~\cite{AnnaswamyFradkov2021}. Adaptive control addresses this
problem by combining parameter estimation with feedback, so that the
control law is updated using observations collected during
operation~\cite{ChenGuo1991}.
For discrete-time linear stochastic systems, classical results establish
stability and asymptotically optimal
tracking under suitable assumptions~\cite{GoodwinRamadgeCaines1981,Guo1995}.
Developing a corresponding theory for nonlinear systems has also
received considerable attention~\cite{BercuPortier2002,LiLam2013}.

When the nonlinearities grow faster than linearly, adaptive control of
discrete-time systems faces not only difficulties in controller design
but also limitations of feedback itself~\cite{Guo2002}.
A precise example is the scalar system
\begin{equation}\label{eq:known_scalar}
 y_{t+1}=\theta y_t^b+u_t+w_{t+1},
\end{equation}
where $y_t$ is the output, $u_t$ is the input, $w_{t+1}$ is the noise,
$\theta$ is an unknown coefficient, and $b>0$ is a known exponent.
Under Gaussian assumptions on the
parameter and noise, Guo~\cite{Guo1997} proved that~\eqref{eq:known_scalar}
is globally stabilizable if and only if $b<4$.
Li and Xie~\cite{LiXie2006} established the same critical exponent
for stabilization robust to bounded noise.

With several unknown drift coefficients, the corresponding model is
\begin{equation}\label{eq:known_multi}
 y_{t+1}=\sum_{i=1}^m\theta_i y_t^{b_i}+u_t+w_{t+1},
\end{equation}
where $m$ is a positive integer, the coefficients $\theta_i$ are unknown,
and the exponents $b_1>1$, $b_1>\cdots>b_m>0$ are known.
Xie and Guo~\cite{XieGuo1999} obtained a polynomial obstruction to
stabilization under Gaussian assumptions, leaving open whether strict
positivity of the polynomial also suffices. This sufficiency question
was resolved under bounded parameter and noise assumptions by Li, Xie,
and Guo~\cite{LiXieGuo2006}, who established a necessary and sufficient
polynomial criterion. Li and Lam~\cite{LiLam2013} completed the Gaussian
case using least-squares (LS) feedback for sufficiency. Under the respective
assumptions, these results characterize stabilizability by
\begin{equation}\label{eq:known_criterion}
 P(x)=x^{m+1}-(x-1)\sum_{i=1}^m b_i x^{m-i}>0,
 \qquad x\in[1,b_1].
\end{equation}
Further results address more general regression functions and certain
systems with faster growth~\cite{LiuLi2019,LiuLi2023}, as well as optimal adaptive output
regulation~\cite{LiuLi2025}.
In the LS analyses of~\cite{LiLam2013,LiuLi2023}, the known
input contribution can be subtracted from the observations, leaving a
regression vector formed from known functions of the output. Estimates
of the associated information matrices and output conditional variances
are central to the stabilization proofs.

These results assume that the input gain is known. For superlinear
systems, uncertainty in
the gain can also change which growth rates permit stabilization.
Joint estimation of the drift coefficients and gain introduces the
feedback input into the regression vector, so the estimates based solely
on output regression functions cannot be applied directly.
For a single unknown drift coefficient, Li and Guo~\cite{LiGuo2010}
considered replacing $u_t$ in~\eqref{eq:known_scalar} by $g u_t$,
where $g\ne0$ is also unknown. Under bounded noise, they showed that
this additional uncertainty reduces the critical exponent from $4$ to $3$.
With several unknown drift coefficients and an unknown input gain, it
remains to determine which joint restrictions on the growth exponents
characterize stabilizability.

We consider the scalar system
\begin{equation}\label{eq:plant}
 y_{t+1}=\sum_{i=1}^m\theta_i y_t^{b_i}+g u_t+w_{t+1},
 \qquad t\ge0,
\end{equation}
where $m$ and the known exponents are as in~\eqref{eq:known_multi}.
The drift coefficients $\theta_i\in\R$ and the input gain $g\ne0$
are fixed and unknown. The noise is iid Gaussian with mean zero and
unknown positive variance.
We establish a necessary and sufficient polynomial criterion for
global stabilizability of~\eqref{eq:plant}. Unlike the single
polynomial condition for a known input gain, this criterion requires
joint positivity of a family of polynomials determined by the growth
exponents. When the criterion holds, the constructed feedback makes the
time average of the squared output converge almost surely to the noise
variance, attaining the smallest possible value.

The main difficulty is that the unknown input gain introduces subsequent
outputs into the interpolation determinants used for parameter calibration,
where terms associated with
different growth exponents can cancel. For sufficiency, we use the
polynomial positivity conditions to choose output ranges in which each
interpolation determinant has a single dominant term. The resulting error
bounds determine calibration inputs that attain the successive output
ranges and the required parameter accuracy. Each calibration is followed
by adaptive feedback, with a performance test triggering a new calibration
upon failure. We prove that only finitely many calibrations occur almost
surely.
For necessity, conditional Gaussian estimates bound output growth under
arbitrary causal feedback. A polynomial identity combines these bounds
to show that failure of the criterion yields output divergence with
positive probability for some fixed parameter values and initial state.

The paper is organized as follows. Section~\ref{sec:results} presents the
main results. Section~\ref{sec:supervisor} introduces the supervisory
controller. Sections~\ref{sec:calibration} and~\ref{sec:local} analyze the
calibration procedure and the adaptive feedback, respectively.
Section~\ref{sec:necessity} proves necessity, and Section~7 concludes the
paper. Supporting proofs are collected in the appendices.

\paragraph{Notation}
The symbol $\|\cdot\|$ denotes the Euclidean norm of a vector and
the induced norm of a matrix, and $\|\cdot\|_p$ denotes the vector
$\ell^p$ norm. For symmetric matrices, $A\succeq B$ means that
$A-B$ is positive semidefinite, and $A\succ B$ means that it is
positive definite. For $A\succ0$, write $\|x\|_A^2=x^\top Ax$.
We write $I_d$ for the $d\times d$ identity matrix, omitting the
subscript when its dimension is clear, and $a\wedge b=\min\{a,b\}$.
The symbols $\ind{B}$ and
$\sgn(x)$ denote the indicator of an event $B$ and the sign of $x$.
In bounds stated almost surely, the constant implicit in $O(\cdot)$
may depend on the sample path, fixed parameters, and initial state.
\section{Main results}\label{sec:results}
For system~\eqref{eq:plant}, let $b=(b_1,\ldots,b_m)$ denote the known
exponent vector, with $b_1>1$ and $b_1>\cdots>b_m>0$. Write
$\theta=(\theta_1,\ldots,\theta_m)^\top$,
$\vartheta=(\theta^\top,g)^\top$, and
$f(y)=(y^{b_1},\ldots,y^{b_m})^\top$.
Assume that each power function $y\mapsto y^{b_i}$ is real-valued on $\R$.
The noise variables $w_t$ are independent and identically distributed with distribution
$N(0,\sigma^2)$, where $\sigma>0$ is unknown.
The drift coefficients and the nonzero input gain are fixed and unknown.
A feedback law has the form $u_t=\pi_t(y_0,\ldots,y_t)$, where each
$\pi_t:\R^{t+1}\to\R$ is Borel measurable.
The observation filtration is $\F_t=\sigma(y_0,\ldots,y_t)$.

We use the following standard notion of stabilizability; see
\cite[Definition~2.1]{LiLam2013}.
\begin{definition}\label{def:stabilizability}
System~\eqref{eq:plant}
is \emph{globally stabilizable} if there exists a feedback such that
\[
 \limsup_{T\to\infty}\frac1T\sum_{t=0}^{T-1}y_{t+1}^2<\infty\quad\as
\]
for every fixed $\theta\in\R^m$, $g\ne0$, $\sigma>0$, and deterministic
initial state $y_0\in\R$.
\end{definition}

For $j=1,\ldots,m$, define
\begin{equation}\label{eq:poly}
 \PP_j(x)=x^{m+2}-(x-1)\left[
 \sum_{i=1}^j b_i x^{m-i+1}+x^{m-j+1}
 +\sum_{i=j+1}^m b_i x^{m-i}\right].
\end{equation}
Since $\PP_j(1)=1$ for every $j$, these polynomials are all strictly
positive on $[1,b_1]$ if and only if none has a zero in that interval.

\begin{theorem}\label{thm:main}
System~\eqref{eq:plant} is globally stabilizable if and only if
\begin{equation}\label{eq:criterion}
 \PP_j(x)>0,\qquad x\in[1,b_1],\quad j=1,\ldots,m.
\end{equation}
\end{theorem}
The sufficiency and necessity parts of Theorem~\ref{thm:main} are proved
in Sections~\ref{sec:supervisor}--\ref{sec:local} and
Section~\ref{sec:necessity}, respectively.

\begin{remark}\label{rem:activeintervals}
Different polynomials determine the criterion on different parts of
$[1,b_1]$. For $j=1,\ldots,m-1$, direct subtraction gives
\begin{equation}\label{eq:adjacent}
 \PP_j(x)-\PP_{j+1}(x)
 =x^{m-j-1}(b_{j+1}-x)(x-1)^2.
\end{equation}
For fixed $x>1$, the sequence $\PP_1(x),\ldots,\PP_m(x)$ therefore
is nonincreasing while $b_{j+1}\ge x$ and nondecreasing once $b_{j+1}\le x$.
Consequently, for $x\in[1,b_1]$, $\PP_j(x)$ attains the minimum among
$\PP_1(x),\ldots,\PP_m(x)$ whenever $b_{j+1}\le x\le b_j$ with
$j<m$, and $\PP_m(x)$ attains this minimum when $x\le b_m$.
Thus it suffices to test each polynomial on its
corresponding interval intersected with $[1,b_1]$, omitting empty
intervals. At a shared endpoint the adjacent polynomials agree.
\end{remark}

\begin{corollary}\label{cor:single}
When $m=1$, system~\eqref{eq:plant} is globally stabilizable if and only if $b_1<3$.
\end{corollary}
\begin{proof}
For $m=1$,
\[
 \PP_1(x)=x\{x^2-(b_1+1)x+b_1+1\}.
\]
The quadratic has positive leading coefficient and discriminant
$(b_1+1)(b_1-3)<0$ when $b_1<3$, so $\PP_1$ is positive on
$[1,b_1]$. For $b_1\ge3$, $2\in[1,b_1]$ and
$\PP_1(2)=2(3-b_1)\le0$. Theorem~\ref{thm:main} gives the result.
\end{proof}

\begin{remark}\label{rem:comparison}
Corollary~\ref{cor:single} gives the critical exponent $3$ under Gaussian
noise, the same threshold established by Li and Guo~\cite{LiGuo2010}
under bounded noise.
\end{remark}

Figure~\ref{fig:region} illustrates the criterion for two drift terms.
The two boundary segments show that different polynomials restrict
different exponent combinations.
\begin{figure}[htbp]
\centering
\includegraphics[width=.95\textwidth]{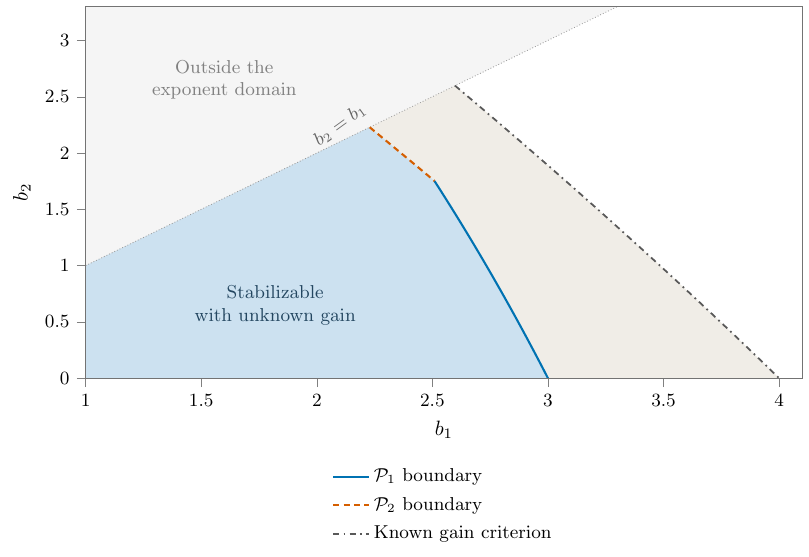}
\caption{Exponent regions for $m=2$ in the domain $1<b_1$ and
$0<b_2<b_1$, using the absolute power model of
Remark~\ref{rem:realexponents} to allow arbitrary real exponents.
The blue region satisfies~\eqref{eq:criterion},
with boundary segments determined by $\PP_1$ (solid) and $\PP_2$
(dashed). The blue and light gray regions together satisfy the known gain
criterion~\eqref{eq:known_criterion} (dash-dotted boundary).
The critical curves are not included in the corresponding stabilizable regions.}
\label{fig:region}
\end{figure}

\begin{corollary}\label{cor:optimality}
If~\eqref{eq:criterion} holds, there is a feedback law depending only on
$b$ such that, for every fixed $\theta\in\R^m$, $g\ne0$, $\sigma>0$,
and deterministic initial state $y_0\in\R$,
\begin{equation}\label{eq:optimalaverage}
 \lim_{T\to\infty}\frac1T\sum_{t=0}^{T-1}y_{t+1}^2=\sigma^2\quad\as.
\end{equation}

For every feedback law, every fixed $\theta\in\R^m$, $g\ne0$,
$\sigma>0$, and deterministic initial state $y_0\in\R$,
\begin{equation}\label{eq:universallower}
 \liminf_{T\to\infty}\frac1T\sum_{t=0}^{T-1}y_{t+1}^2
 \ge\sigma^2\quad\as.
\end{equation}
Thus the constructed feedback attains the smallest possible time average
of the squared output.
\end{corollary}
\begin{proof}
The attainment statement is established in
Sections~\ref{sec:supervisor}--\ref{sec:local}.
For the lower bound, write $y_{t+1}=e_t+w_{t+1}$ with
$e_t$ measurable with respect to $\F_t$.
Applying the exponential martingale argument in the proof of
Lemma~\ref{lem:certificate} to this drift and expanding $y_{t+1}^2$ gives
\[
 \sum_{t=0}^{T-1}y_{t+1}^2
 \ge\tfrac12\sum_{t=0}^{T-1}e_t^2
       +\sum_{t=0}^{T-1}w_{t+1}^2-C_\omega,
\]
where $C_\omega$ is finite almost surely and independent of $T$.
Dividing by $T$, dropping the nonnegative drift term, and applying
the strong law of large numbers proves~\eqref{eq:universallower}.
\end{proof}

\begin{remark}\label{rem:realexponents}
The same criterion holds with $y_t^{b_i}$ replaced
by $|y_t|^{b_i}$, for arbitrary real exponents satisfying
$b_1>1$ and $b_1>\cdots>b_m>0$.
In this form, the exponents need not be restricted by the requirement
that $y^{b_i}$ be real-valued for negative $y$.
The proof is given in Appendix~\ref{app:abspowers}.
\end{remark}
\section{Supervisory control}\label{sec:supervisor}
Under~\eqref{eq:criterion}, we combine finite calibration, adaptive estimation of the parameter ratios,
and supervisory switching~\cite{BorrelliMorse1996}. The performance test
below guarantees stabilization whenever an attempt is retained forever.
Sections~\ref{sec:calibration} and~\ref{sec:local} establish the uniform
success probability used in Proposition~\ref{prop:outer} to show that
switching ends almost surely.

Write the closed-loop drift and its cumulative squared magnitude as
\begin{equation}\label{eq:drift}
 e_t=\theta^\top f(y_t)+g u_t,\qquad
 E_T=\sum_{t=0}^{T-1}e_t^2.
\end{equation}

\subsection{Controller construction}
Each calibration consists of $N$ steps for parameter estimation followed
by two steps of drift cancellation, so its length is $L=N+2$.
Select $N$ and $c_1,\ldots,c_N$ by the rule following
Lemma~\ref{lem:schedule}. They depend only on $b$ and remain fixed
across attempts.
In attempt $k$, the controller multiplies the virtual input by a fixed
sign $s_k\in\{-1,1\}$, so the coefficient of the virtual input is $s_kg$.
The controller estimates the parameter ratio vector $\alpha_k=\theta/(s_kg)$.

Let $R_k$ be the starting time of the $k$th calibration, with $R_1=0$
and $R_k=\infty$ if that calibration is not reached.
Algorithm~\ref{alg:supervisor} specifies calibration, feedback, and the
test whose first failure starts the next attempt.
The first subscript identifies the attempt. The second counts steps
within calibration for $\widetilde y_{k,j}$ and $\widetilde v_{k,j}$,
and within feedback for $x_{k,n}$ and $v_{k,n}$.

\begin{algorithm}[htbp]
\caption{Supervisory feedback}\label{alg:supervisor}
\newcommand{\algorithmformula}[1]{%
  \refstepcounter{equation}%
  \makebox[3em][l]{\textup{(\theequation)}}\(\displaystyle #1\)}
\begin{algorithmic}[1]
\FOR{attempt $k=1,2,\ldots$}\label{ln:supervisor_search}
\STATE\label{ln:supervisor_setup} Set $R_k$ to the current system time, $\widetilde y_{k,0}=y_{R_k}$, and
\par\nobreak\noindent
\algorithmformula{\label{eq:calibrationscale}
 s_k=(-1)^{k-1},\quad H_k=2^k,\quad A_k=k\log(\mathrm e+|y_{R_k}|).
}
\smallskip
\FOR{$j=0,\ldots,L-1$}\label{ln:supervisor_calibrate}
\STATE If $m+2\le j\le N$, solve for $(\widetilde\alpha_{k,j},\widetilde\gamma_{k,j})$ in
\par\nobreak\noindent
\algorithmformula{\label{eq:calinterpolation}
 \widetilde\alpha_{k,j}^\top f(\widetilde y_{k,j-i})
 -\widetilde\gamma_{k,j}\widetilde y_{k,j-i+1}=-\widetilde v_{k,j-i},
 \quad i=1,\ldots,m+1.
}\par\nobreak\noindent
Use $(0,1)$ if the coefficient matrix is singular.
\STATE Compute $\widetilde v_{k,j}$ by
\par\nobreak\noindent
\algorithmformula{\label{eq:calcontrol}
 \widetilde v_{k,j}=\begin{cases}
 e^{A_kc_{j+1}},&0\le j\le m+1,\\
 -\widetilde\alpha_{k,j}^\top f(\widetilde y_{k,j})+e^{A_kc_{j+1}},&m+2\le j<N,\\
 -\widetilde\alpha_{k,N}^\top f(\widetilde y_{k,j}),&j=N,N+1.
 \end{cases}
}
\STATE Apply $u_{R_k+j}=s_k\widetilde v_{k,j}$,
and observe $\widetilde y_{k,j+1}=y_{R_k+j+1}$.
\ENDFOR\label{ln:supervisor_calibration_end}
\smallskip
\STATE\label{ln:supervisor_initialize} Set $x_{k,0}=y_{R_k+L}$.
Initialize $\widehat\alpha_{k,0}=\widetilde\alpha_{k,N}$,
$G_{k,0}=H_kI_{m+1}$, and $z_{k,0}=0$.
\FOR{$n=0,1,2,\ldots$}\label{ln:supervisor_feedback}
\STATE\label{ln:supervisor_control}
\algorithmformula{\label{eq:ce}
 v_{k,n}=-\widehat\alpha_{k,n}^\top f(x_{k,n}),
 \qquad
 \chi_{k,n}=(f(x_{k,n})^\top,v_{k,n})^\top.
}
\STATE\label{ln:supervisor_observe} Apply $u_{R_k+L+n}=s_kv_{k,n}$ and observe $x_{k,n+1}=y_{R_k+L+n+1}$.
\STATE\label{ln:supervisor_monitor} Update
$G_{k,n+1}=G_{k,n}+\chi_{k,n}\chi_{k,n}^\top$ and $z_{k,n+1}=z_{k,n}+\chi_{k,n}x_{k,n+1}$.
\STATE\label{ln:supervisor_update}
Let $F_{k,n+1}$ be the upper left $m\times m$ block of $G_{k,n+1}$ and update
\par\nobreak\noindent
\algorithmformula{\label{eq:ratioupdate}
 \widehat\alpha_{k,n+1}=\widehat\alpha_{k,n}+
 kF_{k,n+1}^{-1}f(x_{k,n})x_{k,n+1}.
}
\STATE\label{ln:supervisor_switch} \textbf{if}
$z_{k,n+1}^\top G_{k,n+1}^{-1}z_{k,n+1}>H_k\log(n+4)$
\textbf{then break}.
\ENDFOR
\ENDFOR
\end{algorithmic}
\end{algorithm}

Alternating $s_k$ lets the controller try both gain signs.
The scale $A_k$ controls the pulse amplitudes, whereas $H_k$
sets the initial information weight and the performance threshold.
The attempt index $k$ also serves as the correction coefficient
in~\eqref{eq:ratioupdate} and remains fixed throughout that attempt.

The initial pulses supply the interpolation observations.
Equation~\eqref{eq:calinterpolation} is formed by
dividing the system equation by $s_kg$ and omitting the noise term.
Here $\widetilde\alpha_{k,j}$ and $\widetilde\gamma_{k,j}$ estimate
$\theta/(s_kg)$ and $1/(s_kg)$, respectively.
Later inputs combine pulses with drift cancellation. The final two
inputs use the terminal estimate without pulses to reduce the output.

\Needspace{4\baselineskip}
During feedback,~\eqref{eq:ce} cancels the estimated drift, giving
\[
 x_{k,n+1}
 =s_kg(\alpha_k-\widehat\alpha_{k,n})^\top f(x_{k,n})+w_{R_k+L+n+1}.
\]

The matrix $F_{k,n}=H_kI_m+\sum_{j=0}^{n-1}f(x_{k,j})f(x_{k,j})^\top$ records information
from the drift regressors. Its updated inverse $F_{k,n+1}^{-1}$ weights
the correction in~\eqref{eq:ratioupdate}, which uses $x_{k,n+1}$ as a noisy
observation of the remaining drift.

The update in~\eqref{eq:ratioupdate} uses the inverse of
$F_{k,n+1}\succeq H_kI_m$. Together with the Borel calibration maps and
observable switching test, this makes
Algorithm~\ref{alg:supervisor} a Borel feedback law.

\subsection{Performance bounds}
For each $k$ with $R_k<\infty$, the feedback interval
$[R_k+L,R_{k+1})$ contains $\nu_k\triangleq R_{k+1}-R_k-L$ steps.
For the output segment $(x_{k,n})_{0\le n\le\nu_k}$, define its drift
and drift energy by
\[
 \begin{aligned}
 e_{k,n}&\triangleq e_{R_k+L+n}=\theta^\top f(x_{k,n})+s_kgv_{k,n},
       &&0\le n<\nu_k,\\
 E_{k,n}&\triangleq\sum_{j=0}^{n-1}e_{k,j}^2
       =E_{R_k+L+n}-E_{R_k+L},
       &&0\le n\le\nu_k.
 \end{aligned}
\]
With $w_{k,n+1}\triangleq w_{R_k+L+n+1}$ and
$\F_{k,n}\triangleq\F_{R_k+L+n}$, its dynamics satisfy
\[
 x_{k,n+1}=e_{k,n}+w_{k,n+1},\qquad
 w_{k,n+1}\mid\F_{k,n}\sim N(0,\sigma^2),
 \quad 0\le n<\nu_k.
\]
Here $\F_{k,0}=\F_{R_k+L}$ represents the information available
at the start of feedback.

The performance test can be written as
\begin{equation}\label{eq:tests}
 J_{k,n}\triangleq z_{k,n}^\top G_{k,n}^{-1}z_{k,n}\le H_k\log(n+3)
\end{equation}
and is checked at local times $1\le n\le\nu_k$.
Its first failure occurs at $n=\nu_k$, with $\nu_k=\infty$ if it
never fails.
The following lemma bounds the drift energy using $J_{k,n}$.

\begin{lemma}\label{lem:certificate}
Let $k\ge1$ and work on $\{R_k<\infty\}$.
Almost surely, the bound
\[
 E_{k,n}\le2J_{k,n}+C_k,\qquad 0\le n\le\nu_k,
\]
holds for every finite integer $n$, where $C_k$ is a finite nonnegative
random variable that does not depend on $n$.
\end{lemma}
\begin{proof}
Completing the square gives
\[
 J_{k,n}\ge2q^\top z_{k,n}-q^\top G_{k,n}q,
 \qquad q\in\R^{m+1}.
\]
For $q=(\theta^\top,s_kg)^\top$, we have
$q^\top\chi_{k,n}=e_{k,n}$ and $\norm q=\norm\vartheta$. Hence
\begin{equation}\label{eq:ridge_true}
 J_{k,n}\ge E_{k,n}+2\sum_{j=0}^{n-1}e_{k,j}w_{k,j+1}-H_k\norm\vartheta^2.
\end{equation}
Since $\nu_k$ is a stopping time for $(\F_{k,n})$ and $e_{k,n}$ is
$\F_{k,n}$-measurable before that time,
\[
 L_{k,n}=\exp\left\{-\frac1{2\sigma^2}\sum_{0\le j<n\wedge\nu_k}e_{k,j}w_{k,j+1}
                   -\frac1{8\sigma^2}E_{k,n\wedge\nu_k}\right\}
\]
is a nonnegative martingale starting at one. By~\eqref{eq:ridge_true},
\[
 E_{k,n}\le2J_{k,n}+2H_k\norm\vartheta^2+8\sigma^2\log L_{k,n},
 \qquad 0\le n\le\nu_k.
\]
The nonnegative martingale convergence theorem gives
$\sup_{n\ge0}L_{k,n}<\infty$ almost surely. Hence the claimed bound holds
with $C_k=2H_k\norm\vartheta^2+8\sigma^2\log(\sup_{n\ge0}L_{k,n})$.
\end{proof}

\begin{lemma}\label{lem:accept}
For either gain sign and every initialization, the following relations
hold as $n\to\infty$ on the event $\{\nu_k=\infty\}$.
\[
 E_{k,n}=O(\log(n+2)),\quad
 \lim_{n\to\infty}\frac1n\sum_{j=0}^{n-1}x_{k,j+1}^2=\sigma^2,\quad
 |x_{k,n}|=O(\sqrt{\log(n+2)})\quad\as
\]
\end{lemma}
\begin{proof}
On $\{\nu_k=\infty\}$, the almost sure bound in
Lemma~\ref{lem:certificate} and~\eqref{eq:tests} give
\[
 E_{k,n}\le2H_k\log(n+3)+C_k=O(\log(n+2))\quad\as
\]
The triangle inequality and $x_{k,j+1}=e_{k,j}+w_{k,j+1}$ give
\[
 \left|\sqrt{\frac1n\sum_{j=0}^{n-1}x_{k,j+1}^2}
       -\sqrt{\frac1n\sum_{j=0}^{n-1}w_{k,j+1}^2}\right|
 \le\sqrt{\frac{E_{k,n}}n}\longrightarrow0.
\]
The strong law of large numbers gives
$n^{-1}\sum_{j=0}^{n-1}w_{k,j+1}^2\to\sigma^2$ almost surely,
so the same limit holds for $n^{-1}\sum_{j=0}^{n-1}x_{k,j+1}^2$.
Gaussian tails and the Borel--Cantelli lemma give
$|w_{k,n}|=O(\sqrt{\log(n+2)})$ almost surely.
Since $e_{k,n-1}^2\le E_{k,n}$, the identity $x_{k,n}=e_{k,n-1}+w_{k,n}$ proves
the stated bound on $x_{k,n}$.
\end{proof}

\subsection{Termination of switching}
\begin{proposition}\label{prop:outer}
Suppose that, for every fixed $(\theta,g,\sigma)$, there is an integer
$k_0$ such that, for every $k\ge k_0$ with $s_k=\sgn g$,
\[
 \P(\nu_k=\infty\mid\F_{R_k})\ge\tfrac38
 \quad\text{a.s. on }\{R_k<\infty\}.
\]
Here $\F_{R_k}$ represents the information available at the start
of the $k$th calibration.
Then, for every fixed parameter and deterministic initial state,
Algorithm~\ref{alg:supervisor} has only finitely many calibrations and
switches almost surely and satisfies
\begin{equation}\label{eq:mainrate}
 E_T=O(\log(T+2)),\quad
 \lim_{T\to\infty}\frac1T\sum_{t=0}^{T-1}y_{t+1}^2=\sigma^2,\quad
 |y_t|=O(\sqrt{\log(t+2)})\quad\as
\end{equation}
\end{proposition}
\begin{proof}
Choose $k_*\ge k_0$ with $s_{k_*}=\sgn g$.
The calibration starting times $R_k$ are stopping times.
For $k=k_*+2j$, $j\ge0$,
reaching $R_{k+2}$ requires the attempt at $R_k$ to fail.
The conditional failure probability is at most $5/8$, so
\[
 \P(R_{k+2}<\infty)
 \le\E\!\left[\ind{R_k<\infty}
       \P(\nu_k<\infty\mid\F_{R_k})\right]
 \le\tfrac58\P(R_k<\infty).
\]
Iteration gives
\begin{equation}\label{eq:trialtail}
 \P(R_{k_*+2j}<\infty)\le(5/8)^j,\qquad j\ge0.
\end{equation}
A failed test is detected at a finite time, and each calibration has
length $L$. Thus~\eqref{eq:trialtail} proves that an attempt is
permanently accepted after finitely many attempts almost surely.

For each $k$ with $R_k<\infty$, condition on $\F_{R_k+L}$ and apply
Lemma~\ref{lem:accept} on $\{\nu_k=\infty\}$.
A countable union of the exceptional null sets is null, so the last
feedback interval satisfies the lemma's conclusions almost surely.
The preceding trajectory has finite length and finite drift energy,
so the same asymptotic conclusions hold for the full trajectory,
giving~\eqref{eq:mainrate}.
\end{proof}
\section{Finite calibration from an arbitrary state}\label{sec:calibration}
We show that increasing the calibration scale in Algorithm~\ref{alg:supervisor}
achieves any prescribed accuracy for the parameter ratios in a fixed number
of steps under bounded disturbances. The terminal output is bounded in
terms of the disturbance bound.
Consider an attempt with $s_k=\sgn g$, so that $s_kg>0$.
In this section, write $g$ for this gain and fix $a\ge2$ such that
$\norm\theta_\infty\le a$ and $a^{-1}\le g\le a$.
We omit the attempt index, write $v_j$ for $\widetilde v_{k,j}$,
and count steps from the start of calibration.

\subsection{Choice of the calibration scales}
We prescribe target magnitudes $e^{Ac_j}$, where $A>0$, to separate
the interpolation samples and let subsequent pulses dominate the residual drift.
For an integer $N\ge m+3$ and a base $q_N\in(0,1)$, define
$c=(c_1,\ldots,c_N)$ by
\begin{equation}\label{eq:schedulechoice}
 c_j=q_N^{(N-j)^2},
 \qquad j=1,\ldots,N.
\end{equation}
Its successive ratios decrease, while $c_N=1$.
To state the conditions for choosing $N$ and $q_N$, define, for $j=1,\ldots,m$
and $m+1\le t\le N$,
\begin{align}
 L_j(c;t)&=\sum_{i=1}^j b_i c_{t-i+1}+c_{t-j+1}
                +\sum_{i=j+1}^m b_i c_{t-i},\label{eq:Lj}\\
 V(c;t)&=\max_{1\le j\le m}L_j(c;t).\label{eq:V}
\end{align}
For $m+1\le t<N$, also set
$L_0(c;t)=c_{t+1}+\sum_{i=1}^m b_i c_{t-i}$.
The distinctness condition below is used to identify a unique dominant
term in the interpolation determinant.

\begin{lemma}\label{lem:schedule}
Assume~\eqref{eq:criterion}. For every $0<\rho<1$, there exist an integer
$N\ge m+3$ and a rational $q_N\in(0,1)$ such that~\eqref{eq:schedulechoice} satisfies
$0<c_1<\cdots<c_N=1$,
\begin{align}
 c_{t+1}&>b_1c_t,&& 1\le t\le m+1,\label{eq:fast}\\
 V(c;t)-V(c;t-1)&<c_{t+1},&& m+2\le t<N,\label{eq:slow}\\
 c_{N-m-1}&>1-\rho\label{eq:terminal}
\end{align}
and, for each $m+1\le t<N$,
\begin{equation}\label{eq:scaledistinct}
 L_p(c;t)\ne L_q(c;t),\qquad 0\le p<q\le m.
\end{equation}
\end{lemma}
\begin{proof}
Condition~\eqref{eq:criterion} implies $b_1<3$. Otherwise,
$\PP_1(2)\le2^m(3-b_1)\le0$.
Take $q_N=e^{-1/N}$. The successive ratios
$r_j=c_j/c_{j-1}$ satisfy
\[
 r_j=e^{[2(N-j)+1]/N},
 \qquad 2\le j\le N.
\]
Thus $1<r_j<e^2$, and for each fixed $j\ge2$,
$\lim_{N\to\infty}r_j=e^2>b_1$, proving~\eqref{eq:fast}.
Also, $c_{N-m-1}=e^{-(m+1)^2/N}$ gives
$\lim_{N\to\infty}c_{N-m-1}=1$, proving~\eqref{eq:terminal}.

To prove~\eqref{eq:slow}, note that~\eqref{eq:adjacent} gives
$\min_{1\le j\le m}\PP_j(r)=\PP_1(r)$ for $r\ge b_1$.
Using $b_i\le b_1$ and the finite geometric sum, we obtain
\[
 \PP_1(r)\ge r^{m-1}(r-b_1)(r^2-r+1)+b_1>0,
 \qquad r\ge b_1.
\]
Together with~\eqref{eq:criterion} and continuity, this gives
$\min_{1\le r\le e^2,\,1\le j\le m}\PP_j(r)/r^{m+1}>0$.

For $h=0,\ldots,m+1$, the identity
$c_{t-h}/c_t=r_t^{-h}e^{-h(h-1)/N}$ gives
$c_{t-h}/c_t=r_t^{-h}+O(N^{-1})$ uniformly for $m+2\le t<N$.
Substitution into~\eqref{eq:Lj} and~\eqref{eq:poly} gives
\[
\begin{aligned}
 \frac{V(c;t)-V(c;t-1)}{c_t}
 &\le\max_{1\le j\le m}
      \frac{L_j(c;t)-L_j(c;t-1)}{c_t}\\
 &=r_t-\min_{1\le j\le m}\frac{\PP_j(r_t)}{r_t^{m+1}}+O(N^{-1}).
\end{aligned}
\]
Since $r_{t+1}=e^{-2/N}r_t=r_t+O(N^{-1})$, the positive minimum
above makes the last expression less than $r_{t+1}$ uniformly in $t$
for all sufficiently large $N$, proving~\eqref{eq:slow}.

Fix such an $N$ and now vary $q_N$. By continuity, the strict inequalities
\eqref{eq:fast}--\eqref{eq:terminal} persist in a neighborhood of
$e^{-1/N}$. For $m+1\le t<N$ and $0\le p<q\le m$, direct subtraction gives
\[
 L_p(c;t)-L_q(c;t)
 =c_{t-p+1}-c_{t-q+1}
   +\sum_{i=p+1}^q b_i(c_{t-i}-c_{t-i+1}).
\]
By~\eqref{eq:schedulechoice}, this is a polynomial in $q_N$ of degree
$(N-t+q)^2$ with leading coefficient $b_q>0$. The finitely many
differences therefore vanish at only finitely many points. Choosing
a rational $q_N$ in this neighborhood outside this finite set
proves~\eqref{eq:scaledistinct} while preserving
\eqref{eq:fast}--\eqref{eq:terminal}.
\end{proof}

We now fix the calibration parameters used in Algorithm~\ref{alg:supervisor}.
Set
\begin{equation}\label{eq:calibrationparameters}
 \eta=\frac{b_m}{8(b_m+2)},\qquad
 \rho=\frac{\eta}{2(1+\sum_{i=1}^m b_i)}.
\end{equation}
Choose the first pair $(N,q_N)$ in a fixed enumeration of
$\{m+3,m+4,\ldots\}\times(\mathbb Q\cap(0,1))$
satisfying~\eqref{eq:fast}--\eqref{eq:scaledistinct}.
Lemma~\ref{lem:schedule} guarantees this choice.
The calibration length is
\begin{equation}\label{eq:selectedlength}
 L=N+2.
\end{equation}
Both $c$ and $L$ depend only on $b$ and are fixed before the first attempt.

\subsection{Interpolation error bounds}
Put $\alpha=\theta/g$, $\gamma=1/g$. Dividing the system equation by $g$ gives
\begin{equation}\label{eq:normreg}
 \alpha^\top f(\widetilde y_s)-\gamma\widetilde y_{s+1}=-v_s-\gamma w_{s+1}.
\end{equation}
For $A>0$, consider the output bounds
\begin{equation}\label{eq:bands}
 (2a)^{-1}e^{Ac_s}\le\widetilde y_s\le2ae^{Ac_s},
 \qquad 1\le s\le N.
\end{equation}
For $m+2\le t\le N$, let $\mathsf M_t$ denote the coefficient matrix
in~\eqref{eq:calinterpolation},
\begin{equation}\label{eq:Mt}
 \mathsf M_t=\begin{pmatrix}
 f(\widetilde y_{t-1})^\top&-\widetilde y_t\\
 f(\widetilde y_{t-2})^\top&-\widetilde y_{t-1}\\
 \vdots&\vdots\\
 f(\widetilde y_{t-m-1})^\top&-\widetilde y_{t-m}
 \end{pmatrix}.
\end{equation}

For nonsingular $\mathsf M_t$, subtracting~\eqref{eq:normreg} from the
interpolation equations gives
\[
 \mathsf M_t\binom{\widetilde\alpha_t-\alpha}{\widetilde\gamma_t-\gamma}
       =\gamma(w_t,\ldots,w_{t-m})^\top.
\]
Write $\beta=1+\sum_{i=1}^m b_i$.

\begin{lemma}\label{lem:cramer}
There is a constant $K_b\ge1$, depending only on $b$, such that the
following bounds hold with $C_a=K_ba^{2\beta+1}$.
For $m+2\le t\le N$ and $A\ge K_b\log a$, assume~\eqref{eq:bands}
for $s=t-m-1,\ldots,t$. Then
\begin{equation}\label{eq:detlower}
 |\det \mathsf M_t|\ge C_a^{-1}e^{A V(c;t-1)}>0.
\end{equation}
If $\max_{t-m\le s\le t}|w_s|\le W$ with $W\ge1$, then
\begin{equation}\label{eq:currenterr}
 |(\widetilde\alpha_t-\alpha)^\top f(\widetilde y_t)|
 \le C_aW e^{A[V(c;t)-V(c;t-1)]}.
\end{equation}
For $t=N$, also
\begin{equation}\label{eq:coeferr}
 (2a)^{b_i}|\widetilde\alpha_{i,N}-\alpha_i|
 \le C_aW e^{-(b_i-\eta/2)A},\qquad i=1,\ldots,m.
\end{equation}
\end{lemma}
\begin{proof}
Put $q=2a$ and expand $\det \mathsf M_t$ along its last column.
Selecting $-\widetilde y_{t-p}$ in row $p+1$, with $0\le p\le m$,
gives products of magnitude
\[
 \widetilde y_{t-p}\prod_{i=1}^m\widetilde y_{s_i}^{b_i},
 \qquad \lambda=c_{t-p}+\sum_{i=1}^m b_ic_{s_i},
\]
where $(s_1,\ldots,s_m)$ permutes
$\{t-m-1,\ldots,t-1\}\setminus\{t-p-1\}$.
By~\eqref{eq:bands}, each product lies between
$q^{-\beta}e^{A\lambda}$ and $q^\beta e^{A\lambda}$.
Pairing $b_1>\cdots>b_m$ with the remaining $c_s$ in decreasing order
uniquely maximizes $\lambda$, since exchanging $s_i$ and $s_j$
when $i<j$ and $s_i<s_j$ increases it by
$(b_i-b_j)(c_{s_j}-c_{s_i})>0$.
This maximum is $L_p(c;t-1)$.
The distinctness in Lemma~\ref{lem:schedule} therefore gives a unique
largest exponent in $\det \mathsf M_t$. Since $m+2\le t\le N$ ranges over a
finite set, its gap from all other exponents has a positive lower bound
$\delta$ depending only on $b$.
The ratio of the sum of the absolute values of all other terms to the
absolute value of the leading term is at most
\[
 ((m+1)!-1)(2a)^{2\beta}e^{-\delta A}\le\tfrac12
\]
when $A\ge K_b\log a$, with $K_b$ chosen sufficiently large.
Thus
\[
 |\det \mathsf M_t|\ge\tfrac12q^{-\beta}e^{A\max_{0\le p\le m}L_p(c;t-1)}.
\]
This proves~\eqref{eq:detlower}, since
$\max_{0\le p\le m}L_p(c;t-1)\ge V(c;t-1)$.

For $i=1,\ldots,m$, let $\mathsf M_t^{(i)}$ be $\mathsf M_t$ with column $i$
replaced by $(w_t,\ldots,w_{t-m})^\top$. Cramer's rule gives
\[
 \widetilde\alpha_{i,t}-\alpha_i
 =\gamma\frac{\det \mathsf M_t^{(i)}}{\det \mathsf M_t}.
\]
Label the row $(f(\widetilde y_s)^\top,-\widetilde y_{s+1})$ by $s$.
In each term of $\widetilde y_t^{b_i}\det \mathsf M_t^{(i)}$, omit the noise
factor and assign $\widetilde y_t^{b_i}$ the label $t$. The resulting $m+1$
row labels, in decreasing order, are bounded by $t,t-1,\ldots,t-m$.
Replacing them by these bounds cannot decrease the exponent, since the $c_s$
increase and all $b_j>0$. Label $t$ supplies a power factor,
so the last column has label $t-j$ for some $1\le j\le m$ and contributes
$c_{t-j+1}$ to the exponent.
Rearrangement therefore bounds the exponent by $L_j(c;t)\le V(c;t)$.
Since $\det \mathsf M_t^{(i)}$ has at most $(m+1)!$ terms,
\[
 \widetilde y_t^{b_i}|\det \mathsf M_t^{(i)}|
 \le (m+1)!\,Wq^\beta e^{AV(c;t)}.
\]
Summing over $i$ and using $\gamma\le a$ and
$|\det \mathsf M_t|\ge\tfrac12q^{-\beta}e^{AV(c;t-1)}$
prove~\eqref{eq:currenterr} provided
$K_b\ge2^{2\beta+1}(m+2)!$.

At $t=N$, condition~\eqref{eq:terminal} gives
$1-\rho\le c_s\le1$ for $N-m-1\le s\le N$.
The powers in each term of $\det \mathsf M_N^{(i)}$ sum to $\beta-b_i$.
The same Cramer formula and $\rho\beta=\eta/2$ give
\[
 q^{b_i}|\widetilde\alpha_{i,N}-\alpha_i|
 \le\frac{a(m+1)!\,Wq^\beta e^{A(\beta-b_i)}}
              {\tfrac12q^{-\beta}e^{A(1-\rho)\beta}}
 \le C_aW e^{-A(b_i-\eta/2)}.
\]
This proves~\eqref{eq:coeferr} with the same $C_a$.
\end{proof}

\subsection{Calibration bounds}\label{sec:calibrationconstruction}
Under the hypotheses of Lemma~\ref{lem:cramer},
\eqref{eq:currenterr} and~\eqref{eq:slow} give, as $A\to\infty$,
\[
 \frac{|(\widetilde\alpha_t-\alpha)^\top f(\widetilde y_t)|}{e^{Ac_{t+1}}}
 \le C_aW e^{-A[c_{t+1}-V(c;t)+V(c;t-1)]}\longrightarrow0,
 \qquad m+2\le t<N.
\]
We use this exponent gap to establish~\eqref{eq:bands} inductively
and then apply the two cancellation inputs.

\begin{proposition}\label{prop:calibration}
Assume~\eqref{eq:criterion}, and take $L$ from~\eqref{eq:selectedlength}.
There is a constant $C_b>0$, depending only on $b$, such that for every
$a\ge2$, $W\ge1$, $\varepsilon>0$, finite initial state $\widetilde y_0$,
and parameters with $\norm\theta_\infty\le a$ and $a^{-1}\le g\le a$,
the calibration law~\eqref{eq:calcontrol} with
\begin{equation}\label{eq:calibrationbound}
 A\ge C_b\log\left[
 (\mathrm e+a+W+\varepsilon^{-1})(\mathrm e+|\widetilde y_0|)\right]
\end{equation}
yields, if $\max_{1\le s\le L}|w_s|\le W$,
\begin{equation}\label{eq:seed}
 |\widetilde y_L|\le W+1,\qquad
 \norm{\widetilde\alpha_N-\alpha}\le\varepsilon.
\end{equation}
The calibration inputs remain finite on every finite output history.
\end{proposition}
\begin{proof}
By~\eqref{eq:fast} and~\eqref{eq:slow}, choose
$0<\mu\le\min\{c_1,\eta\}$, depending only on $b$, such that
\[
 \begin{aligned}
 b_1c_j+\mu&<c_{j+1},&&1\le j\le m+1,\\
 V(c;j)-V(c;j-1)+\mu&<c_{j+1},&&m+2\le j<N.
 \end{aligned}
\]
Put $R=e^{\mu A}$.
With $C_a$ as in Lemma~\ref{lem:cramer}, choose $C_b$ large enough that
\eqref{eq:calibrationbound} implies $A\ge K_b\log a$ and
\[
 R\ge4ma^2C_aW,\qquad
 R\ge2a\left(a\|f(\widetilde y_0)\|_1+W\right),\qquad
 R\ge\frac{\sqrt m\,C_aW}{\varepsilon}.
\]
Taking logarithms of the three bounds and using $C_a=K_ba^{2\beta+1}$
and $1+\|f(y)\|_1\le(m+1)(\mathrm e+|y|)^{b_1}$ shows that $C_b$
can be chosen to depend only on $b$.
For the first pulse, $\mu\le c_1$ gives
\[
 a\sum_{i=1}^m|f_i(\widetilde y_0)|+W\le\frac1{2a}e^{Ac_1}.
\]
Hence~\eqref{eq:bands} holds for $s=1$.
If~\eqref{eq:bands} holds through step $j$, then for $1\le j\le m+1$,
\[
 \frac{|\theta^\top f(\widetilde y_j)+w_{j+1}|}{e^{Ac_{j+1}}}
 \le\frac{am(2a)^{b_1}+W}{R}\le\frac1{2a}.
\]
For $m+2\le j<N$, the samples $\widetilde y_{j-m-1},\ldots,\widetilde y_j$
satisfy~\eqref{eq:bands} by the induction hypothesis.
Applying~\eqref{eq:currenterr} at $t=j$ gives
\[
 \frac{|(\theta-g\widetilde\alpha_j)^\top
 f(\widetilde y_j)+w_{j+1}|}{e^{Ac_{j+1}}}
 \le\frac{(aC_a+1)W}{R}\le\frac1{2a}.
\]
These bounds use the choice of $\mu$ and $C_a\ge(2a)^{b_1}$.
Together with $a^{-1}\le g\le a$, these bounds give~\eqref{eq:bands}
at $s=j+1$ in both cases, completing the induction for $1\le s\le N$.

For the final two inputs, equation~\eqref{eq:calcontrol} gives
\[
 \widetilde y_{j+1}
 =g(\alpha-\widetilde\alpha_N)^\top f(\widetilde y_j)+w_{j+1},
 \qquad j=N,N+1.
\]
Using~\eqref{eq:bands} with $c_N=1$ and~\eqref{eq:coeferr},
\[
 |\widetilde y_{N+1}|
 \le a\sum_{i=1}^m|\widetilde\alpha_{i,N}-\alpha_i|(2ae^A)^{b_i}+W
 \le(amC_a+1)W e^{\eta A/2}\le e^{2\eta A},
\]
because $(amC_a+1)W\le R\le e^{\eta A}$.
For the second cancellation input,~\eqref{eq:calibrationparameters} gives
$(1-2\eta)b_m-\eta/2>\eta\ge\mu$. Therefore
\[
 \begin{aligned}
 |\widetilde y_{N+2}|
 &\le a\sum_{i=1}^m|\widetilde\alpha_{i,N}-\alpha_i|
                   |\widetilde y_{N+1}|^{b_i}+W\\
 &\le amC_aW e^{-[(1-2\eta)b_m-\eta/2]A}+W
 \le\frac{amC_aW}{R}+W\le W+1.
 \end{aligned}
\]

Since $\mu\le\eta<b_m-\eta/2$, equation~\eqref{eq:coeferr} bounds every
coordinate error by $C_aW/R$. Consequently,
\[
 \norm{\widetilde\alpha_N-\alpha}
 \le\frac{\sqrt m\,C_aW}{R}\le\varepsilon.
\]

With $(0,1)$ used when an interpolation matrix is singular, all
$L=N+2$ calibration inputs are finite on every finite output history.
\end{proof}
\section{Local analysis of the feedback}\label{sec:local}
Fix $a_0\ge2$ such that
$\norm\theta_\infty\le a_0$ and $a_0^{-1}\le|g|,\sigma\le a_0$.
Fix an attempt $k\ge a_0$ with $s_k=\sgn g$ and $R_k<\infty$.
The local probability bounds are conditional on $\F_{R_k+L}$.
Continue the feedback, matrix updates, and parameter update in
Algorithm~\ref{alg:supervisor} for all $n\ge0$, with $H_k=2^k$,
ignoring the stopping test. In particular, continue $G_{k,n}$ and
$z_{k,n}$ and define $J_{k,n}$ by~\eqref{eq:tests}.
This recurrence agrees with the actual feedback until the first failed
performance test.
As shown in the proof of Lemma~\ref{lem:schedule},
\eqref{eq:criterion} implies $b_1<3$, so
\begin{equation}\label{eq:upper}
 \norm{f(y)}^2\le m\max\{1,|y|^6\},\quad y\in\R.
\end{equation}
The ratio error and drift satisfy
\begin{equation}\label{eq:ratioerror}
 \delta_{k,n}=\alpha_k-\widehat\alpha_{k,n},\qquad
 e_{k,n}=s_kg\delta_{k,n}^\top f(x_{k,n}),\qquad x_{k,n+1}=e_{k,n}+w_{k,n+1}.
\end{equation}
Write $E_{k,n}=\sum_{j=0}^{n-1}e_{k,j}^2$ for its cumulative squared drift.
For the information matrix $F_{k,n}$ in Algorithm~\ref{alg:supervisor}, define
\begin{align}
 U_{k,n}&=\frac{s_kg}{k}\delta_{k,n}^\top F_{k,n}\delta_{k,n},\qquad
 q_{k,n}=f(x_{k,n})^\top F_{k,n}^{-1}f(x_{k,n}),\label{eq:ratiopotential}\\
 D_{k,n}^F&=\log\frac{\det F_{k,n}}{H_k^m},\qquad
 D_{k,n}^G=\log\frac{\det G_{k,n}}{H_k^{m+1}},\label{eq:localdet}\\
 \tau_k&=\inf\{n\ge0: |e_{k,n}|>1/2\ \text{or}\ q_{k,n}>(2k^2)^{-1}\}.
 \label{eq:localstop}
\end{align}
Set $M_{k,n}=\sum_{j=0}^{n-1}\chi_{k,j}w_{k,j+1}$.
The decomposition of $z_{k,n}$ and the Cauchy--Schwarz inequality, using
$G_{k,n}\succeq\sum_{j=0}^{n-1}\chi_{k,j}\chi_{k,j}^\top$, give, for $n\ge0$,
\[
 \begin{gathered}
 z_{k,n}=\sum_{j=0}^{n-1}\chi_{k,j}e_{k,j}+M_{k,n},\qquad
 \norm{z_{k,n}-M_{k,n}}_{G_{k,n}^{-1}}^2\le E_{k,n},\\
 J_{k,n}\le2E_{k,n}+2\norm{M_{k,n}}_{G_{k,n}^{-1}}^2.
 \end{gathered}
\]
To bound $E_{k,n}$ for all $n$, we estimate $U_{k,n}$ and $E_{k,n}$ up to
$\tau_k$ and use observations $x_{k,s}\in[0,1]$ to exclude a finite
$\tau_k$ on an event of positive probability.
Related Gaussian martingale tools for stochastic LS and self-normalized processes
appear in~\cite{LaiWei1982,Abbasi2011}.
\begin{lemma}\label{lem:energy}
Fix an attempt $k\ge a_0$ with $s_k=\sgn g$, and let $0<\epsilon<1/2$.
With probability at least $1-\epsilon$,
simultaneously for all $0\le n\le\tau_k$,
\begin{equation}\label{eq:energy}
 U_{k,n}+\tfrac14E_{k,n}
 \le U_{k,0}+2k^2\sigma^2D_{k,n}^F
       +16\sigma^2\log(1/\epsilon).
\end{equation}
With probability at least $1-\epsilon$, simultaneously for all $n\ge0$,
\begin{equation}\label{eq:selfnorm}
 \norm{M_{k,n}}_{G_{k,n}^{-1}}^2
 \le\sigma^2\{D_{k,n}^G+2\log(1/\epsilon)\}.
\end{equation}
\end{lemma}
\begin{proof}
Put $h_{k,n}=q_{k,n}/(1+q_{k,n})$ for $n\ge0$ and
$\lambda=k s_kg\in[1,k^2]$.
By~\eqref{eq:ratioupdate} and the Sherman--Morrison formula,
\[
 \begin{gathered}
 \delta_{k,n+1}=\delta_{k,n}
   -kF_{k,n+1}^{-1}f(x_{k,n})(e_{k,n}+w_{k,n+1}),\\
 f(x_{k,n})^\top F_{k,n+1}^{-1}f(x_{k,n})
   =\frac{q_{k,n}}{1+q_{k,n}}=h_{k,n}.
 \end{gathered}
\]
Substituting the error update into~\eqref{eq:ratiopotential} gives
\[
 U_{k,n+1}-U_{k,n}
 =(\lambda^{-1}-2+\lambda h_{k,n})e_{k,n}^2
   -2(1-\lambda h_{k,n})e_{k,n}w_{k,n+1}+\lambda h_{k,n}w_{k,n+1}^2.
\]
Set $\beta_{k,n}=2(1-\lambda h_{k,n})$.
For $n<\tau_k$, one has $\lambda h_{k,n}\le k^2q_{k,n}\le1/2$,
$\beta_{k,n}\in[1,2]$, and
$\lambda^{-1}-2+\lambda h_{k,n}\le-1/2$. Hence
\[
 U_{k,n+1}-U_{k,n}
 \le-\tfrac12e_{k,n}^2-\beta_{k,n}e_{k,n}w_{k,n+1}+k^2h_{k,n}w_{k,n+1}^2.
\]
For $Z\sim N(0,1)$, $u\in\R$, and $0\le v\le1/4$, Gaussian integration gives
\[
 \E e^{uZ+vZ^2}
 =(1-2v)^{-1/2}\exp\!\left\{\frac{u^2}{2(1-2v)}\right\}
 \le e^{u^2+2v},
\]
using $1-2v\ge1/2$ and $-\log(1-2v)\le4v$.
For $n<\tau_k$, apply this conditionally with
$u=-\beta_{k,n}e_{k,n}/(16\sigma)$ and $v=k^2h_{k,n}/16$.
These coefficients are $\F_{k,n}$-measurable and satisfy
$u^2\le e_{k,n}^2/(64\sigma^2)$ and $v\le1/32$.
The matrix determinant lemma gives
$D_{k,n+1}^F-D_{k,n}^F=\log(1+q_{k,n})\ge h_{k,n}$.
The increment bound therefore shows that
\[
 \exp\left\{\frac{U_{k,n}+\tfrac14E_{k,n}-U_{k,0}}{16\sigma^2}
             -\frac{k^2}{8}D_{k,n}^F\right\},\qquad n\ge0,
\]
stopped at $\tau_k$, is a nonnegative supermartingale starting at one.
Its maximal inequality at level $1/\epsilon$ proves~\eqref{eq:energy}.

For~\eqref{eq:selfnorm}, apply~\cite[Theorem~1]{Abbasi2011}
with initial matrix $H_kI_{m+1}$, regressors $\chi_{k,j}$,
and noise $w_{k,j+1}$.
The regressors are $\F_{k,j}$-measurable and
$w_{k,j+1}\mid\F_{k,j}\sim N(0,\sigma^2)$.
\end{proof}

For the power functions, take $J=[0,1]$. Their Gram matrix is
\begin{equation}\label{eq:localgram}
 \Gamma_J=\int_J f(x)f(x)^\top\dd x
 =\left(\frac{1}{b_i+b_j+1}\right)_{i,j=1}^m\succ0.
\end{equation}
Indeed, if $v\in\R^m$ satisfies $v^\top\Gamma_Jv=0$, then
$\sum_{i=1}^m v_i x^{b_i}=0$ on $(0,1)$ by continuity. Dividing by $x^{b_m}$
and letting $x\downarrow0$ gives $v_m=0$. Repeating with the remaining
exponents gives $v=0$.
Set
\[
 \kappa=\frac{\exp(-9a_0^2/8)}{\sqrt{2\pi}\,a_0}
          \lambda_{\min}(\Gamma_J),
\]
so that $0<\kappa<1$.

\begin{lemma}\label{lem:local}
Let $\norm\theta_\infty\le a_0$ and $a_0^{-1}\le|g|,\sigma\le a_0$.
There is $k_0=k_0(b,a_0)\ge a_0$ such that, for every $k\ge k_0$
with $s_k=\sgn g$, on $\{R_k<\infty\}$ and for initializations satisfying
\begin{equation}\label{eq:localinit}
 |x_{k,0}|\le k+1,\qquad
 \norm{\widehat\alpha_{k,0}-\alpha_k}\le k^{-6},
\end{equation}
the continued feedback satisfies, almost surely,
\begin{equation}\label{eq:localmonitor}
 \P\!\left(J_{k,n}\le H_k\log(n+3)\ \text{for all }n\ge0
       \,\middle|\,\F_{R_k+L}\right)\ge\tfrac12.
\end{equation}
\end{lemma}
\begin{proof}
The constants in $O(\cdot)$ below depend only on $b,a_0$.
For $s\ge1$, define
\[
 Z_{k,s}=\ind{s\le\tau_k,\ x_{k,s}\in J}f(x_{k,s})f(x_{k,s})^\top,\qquad
 \Xi_{k,s}=Z_{k,s}-\E[Z_{k,s}\mid\F_{k,s-1}].
\]
Since $\{\tau_k\ge s\}\in\F_{k,s-1}$ and
$x_{k,s}\mid\F_{k,s-1}\sim N(e_{k,s-1},\sigma^2)$, the bound
$|e_{k,s-1}|\le1/2$ on $\{\tau_k\ge s\}$ gives, by integration over $J$,
\begin{equation}\label{eq:stoppedmean}
 \E[Z_{k,s}\mid\F_{k,s-1}]
 \succeq\frac{\exp(-9a_0^2/8)}{\sqrt{2\pi}\,a_0}\ind{\tau_k\ge s}\Gamma_J
 \succeq\kappa\ind{\tau_k\ge s}I_m.
\end{equation}
Each entry of $Z_{k,s}$ lies in $[0,1]$, so each entry of $\Xi_{k,s}$
is a martingale difference bounded in absolute value by one.
Write $\norm{\cdot}_{\mathrm F}$ for the Frobenius norm.
The Azuma--Hoeffding inequality~\cite[Lemma~1]{Azuma1967} and a union bound
over the $m^2$ entries and $n\ge1$ give
\begin{equation}\label{eq:stoppeddeviation}
 \P\left(\exists n\ge1:\,
 \Bigl\|\sum_{s=1}^n\Xi_{k,s}\Bigr\|_{\mathrm F}
 >\frac{H_k+\kappa n}{4}\right)
 \le2m^2\sum_{n=1}^{\infty}
 \exp\!\left\{-\frac{(H_k+\kappa n)^2}{32m^2n}\right\}.
\end{equation}
Since $(H_k+\kappa n)^2/n\ge2H_k\kappa+\kappa^2n$, the right-hand side
is $O(e^{-H_k\kappa/(16m^2)})$, hence at most $1/12$ for all sufficiently
large $k$. On the complementary event in~\eqref{eq:stoppeddeviation},
\eqref{eq:stoppedmean} and $\norm{\cdot}\le\norm{\cdot}_{\mathrm F}$ give
$F_{k,n}\succeq\tfrac34[H_k+\kappa(n-1)]I_m$ for $1\le n\le\tau_k$.
Since $H_k\ge2\kappa$ and $F_{k,0}=H_kI_m$, we obtain
\begin{equation}\label{eq:goodgram}
 F_{k,n}\succeq\tfrac12(H_k+\kappa n)I_m,\qquad0\le n\le\tau_k.
\end{equation}

The Gaussian tail bound gives, with probability at least $11/12$,
\begin{equation}\label{eq:noisetube}
 |w_{k,n+1}|\le4k\sqrt{\log(n+3)},\qquad n\ge0.
\end{equation}
The failure probability is at most
$2\sum_{n\ge0}{(n+3)^{-8}}\le1/12$.
By a union bound, \eqref{eq:goodgram}, \eqref{eq:noisetube},
and \eqref{eq:energy}--\eqref{eq:selfnorm} with $\epsilon=1/6$
hold simultaneously with probability at least
$1-1/12-1/12-1/6-1/6=1/2$.

On this event, put $\ell_n=\log(n+3)>1$. For $n\le\tau_k$,
the initial bound $|x_{k,0}|\le k+1$ and the inequalities
$|e_{k,n-1}|\le1/2$ for $1\le n\le\tau_k$ give
\[
 |x_{k,n}|\le5k\sqrt{\ell_n},\qquad
 \norm{f(x_{k,n})}^2=O(k^6\ell_n^3).
\]
For $0\le j<\tau_k$, the identity
$v_{k,j}=(e_{k,j}-\theta^\top f(x_{k,j}))/(s_kg)$,
the bound $|e_{k,j}|\le1/2$, and the parameter bounds give
$\norm{\chi_{k,j}}^2=O(k^6\ell_j^3)$.
Thus, for a constant $C$ depending only on $b,a_0$,
\[
 \begin{aligned}
 F_{k,n}&\preceq(H_k+Cnk^6\ell_n^3)I_m,\\
 G_{k,n}&\preceq(H_k+Cnk^6\ell_n^3)I_{m+1},
 \qquad 0\le n\le\tau_k.
 \end{aligned}
\]
For large $k$, $Ck^6/H_k\le1$, so
\begin{equation}\label{eq:logdetlocal}
 D_{k,n}^F+D_{k,n}^G\le(2m+1)\log(1+n\ell_n^3)=O(\ell_n),
 \qquad n\le\tau_k.
\end{equation}
By~\eqref{eq:energy} and $\sigma\le a_0$, for $n\le\tau_k$,
\begin{equation}\label{eq:localenergy}
 U_{k,n}+\tfrac14E_{k,n}
 \le U_{k,0}+O(k^2\ell_n).
\end{equation}
By~\eqref{eq:localinit},
$U_{k,0}=(s_kg/k)H_k\norm{\delta_{k,0}}^2\le a_0H_k k^{-13}$.
Combining~\eqref{eq:goodgram} and~\eqref{eq:localenergy} therefore gives
\begin{equation}\label{eq:qelocal}
 q_{k,n}=O\!\left(\frac{k^6\ell_n^3}{H_k+\kappa n}\right),\qquad
 e_{k,n}^2\le k s_kgU_{k,n}q_{k,n}
 =O\!\left(\frac{H_k k^{-6}\ell_n^3+k^9\ell_n^4}{H_k+\kappa n}\right).
\end{equation}
For $p=3,4$ and large $k$, the function
$r\mapsto[\log(H_k/\kappa+3)+\log r]^p/r$ is nonincreasing on $[1,\infty)$.
Taking $r=1+\kappa n/H_k$ and using $n+3\le(H_k/\kappa+3)r$ gives
\[
 \sup_{n\ge0}\frac{\ell_n^p}{H_k+\kappa n}
 \le\frac{\log^p(H_k/\kappa+3)}{H_k}
 =O(k^p2^{-k}),\qquad p=3,4.
\]
Substitution into~\eqref{eq:qelocal} yields
\[
 \begin{aligned}
 \sup_{0\le n\le\tau_k}q_{k,n}&=O(k^9 2^{-k})=o(k^{-2}),\\
 \sup_{0\le n\le\tau_k}e_{k,n}^2&=O(k^{-3}+k^{13}2^{-k})=o(1).
 \end{aligned}
\]
For large $k$, these bounds are strictly below $(2k^2)^{-1}$ and $1/4$,
respectively, so $\tau_k=\infty$.
Combining~\eqref{eq:localenergy}, \eqref{eq:logdetlocal}, and~\eqref{eq:selfnorm}
with the bound on $J_{k,n}$ gives
$J_{k,n}\le8U_{k,0}+O(k^2\ell_n)$ for all $n\ge0$. Hence
\[
 \sup_{n\ge0}\frac{J_{k,n}}{H_k\ell_n}
 =O(k^{-13}+k^2 2^{-k})\longrightarrow0,
\]
which proves~\eqref{eq:localmonitor}.
\end{proof}

\begin{proof}[Proof of sufficiency in Theorem~\ref{thm:main}]
With $a_0$ fixed at the start of this section,
consider an attempt $k\ge a_0$ with $s_k=(-1)^{k-1}=\sgn g$
and $R_k<\infty$. Then $\norm\theta_\infty\le k$ and
$k^{-1}\le s_kg=|g|\le k$, so we apply the calibration bound with
$a=W=k$ and $\varepsilon=k^{-6}$.
Since $\log(\mathrm e+|y_{R_k}|)\ge1$, the choice
$A_k=k\log(\mathrm e+|y_{R_k}|)$ satisfies~\eqref{eq:calibrationbound}
for every $y_{R_k}$ as soon as
\[
 k\ge C_b\bigl[1+\log(\mathrm e+2k+k^6)\bigr].
\]
The right-hand side is $O(\log k)$, so this holds for all sufficiently
large $k$, independently of the starting output.
On $\mathcal C=\{\max_{1\le j\le L}|w_{R_k+j}|\le k\}$,
Proposition~\ref{prop:calibration} gives~\eqref{eq:localinit}
at time $R_k+L$.

The Gaussian tail bound gives
\[
 \P(\mathcal C^c\mid\F_{R_k})
 \le2L e^{-k^2/(2a_0^2)}\le\tfrac14
\]
for all sufficiently large $k$.
Future innovations are independent of $\F_{R_k+L}$, so
Lemma~\ref{lem:local} gives
$\P(\nu_k=\infty\mid\F_{R_k+L})\ge1/2$ on $\mathcal C$.
Consequently,
\[
 \P(\nu_k=\infty\mid\F_{R_k})
 \ge\tfrac12\P(\mathcal C\mid\F_{R_k})\ge\tfrac38.
\]
All thresholds on $k$ depend only on $b,a_0$, so this bound is uniform
over output histories at the start of the attempt.
Proposition~\ref{prop:outer} completes the proof.
\end{proof}
\section{Necessity of the polynomial criterion}\label{sec:necessity}
We prove the converse in Theorem~\ref{thm:main}.
More precisely, if the criterion fails, then for every fixed $\sigma>0$,
every Borel feedback, and every compact parameter rectangle
$\Theta\subset\R^{m+1}$ with nonempty interior and gain interval excluding
zero, there are a deterministic initial state and a fixed
$\vartheta\in\Theta$ for which
\begin{equation}\label{eq:escape}
 \P_\vartheta\!\left\{\lim_{t\to\infty}|y_t|=\infty,\quad
 \lim_{T\to\infty}T^{-1}\sum_{t=0}^{T-1}y_{t+1}^2=\infty\right\}>0.
\end{equation}
On an event of positive probability, we combine lower bounds on cumulative
and individual output logarithms through a polynomial identity to force
divergence.

The change of scale $y_t=\sigma\bar y_t$, $u_t=\sigma\bar u_t$
replaces $\theta_i$ by $\theta_i\sigma^{b_i-1}$ and reduces the noise
variance to one. It preserves the feedback class and the rectangle
properties. Work henceforth with this normalized model. Fix a compact
rectangle $\Theta$ with the properties stated above. On an auxiliary space,
draw $\vartheta\sim N(0,I_{m+1})$ once, independently of iid
$w_s\sim N(0,1)$. For any finite Borel feedback and deterministic $y_0$,
the system recursively defines finite outputs and inputs at every time.
All conditional distributions below refer to this unconditioned auxiliary
model.

For a fixed $W>0$, let
\begin{equation}\label{eq:gaussianenvelope}
 \Omega=\{\vartheta\in\Theta,\ |w_s|\le Ws\text{ for every }s\ge1\}.
\end{equation}
If $\Phi$ is the standard normal distribution function, independence yields
\begin{equation}\label{eq:envelopeprob}
 p_0\triangleq\P(\Omega)=\P\{\vartheta\in\Theta\}
                 \prod_{s\ge1}\{2\Phi(Ws)-1\}>0.
\end{equation}
$p_0$ is independent of the feedback and initial state.

Write $Y_t=|y_t|$ and $X_t=\log Y_t$ for nonzero outputs at $t\ge0$.
Set $X_t=0$ for $t<0$ and define
\[
 S_t=\sum_{s=0}^tX_s,\quad t\ge0,\qquad S_t=0,\quad t<0,
 \qquad D_t=X_t-X_{t-1}.
\]
Thus $D_t=\log(Y_t/Y_{t-1})$ for $t\ge1$.
Set $b_i=0$ for $i>m$.
For $1\le j\le m$ and $t\ge0$, define the residuals
\begin{equation}\label{eq:necessresiduals}
 \begin{split}
 R_{j,t}&=S_{t+1}-\sum_{i=1}^j b_iX_{t-i+1}-X_{t-j+1}
                   -\sum_{i=j+1}^m b_iX_{t-i},\\
 W_t&=X_{t+1}-\sum_{i=1}^m b_iD_{t-i+1}.
 \end{split}
\end{equation}

\begin{lemma}\label{lem:necesssample}
There are $K>1$ and $C>0$, depending only on $m,b,\Theta,W$, such that
every finite Borel feedback and deterministic initial state with $Y_0\ge1$
admits an event $\mathcal D\subset\Omega$ with $\P(\mathcal D)\ge p_0/2$.
On $\mathcal D$, all outputs are nonzero. For any $t\ge0$, if
\begin{equation}\label{eq:necessseparation}
 Y_0\ge1,\qquad Y_s\ge KY_{s-1},\quad1\le s\le t,
\end{equation}
then, with $\ell_t=C\{(t+1)^2+\log(2+S_t)\}$,
\begin{equation}\label{eq:necessbothresiduals}
 R_{j,t}\ge-\ell_t,\quad1\le j\le m,\qquad W_t\ge-\ell_t.
\end{equation}
\end{lemma}
The bound on $W_t$ follows from a conditional variance estimate.
The bound on $R_{j,t}$ also uses the conditional mean through the second
moment. Gaussian small ball estimates turn these moment bounds into
pathwise bounds on the same event $\mathcal D$.
Appendix~\ref{app:centers} gives the determinant estimates, the construction
of $\mathcal D$, and the proof of Lemma~\ref{lem:necesssample}.

We combine the residual bounds to propagate $D_t\ge\log K$.
The following identity gives positive coefficients to both residuals
and the past increments.
\begin{lemma}\label{lem:necessroot}
If~\eqref{eq:criterion} fails, there exist $1<r<b_1$ and
$j\in\{1,\ldots,m\}$ such that
\[
 \PP_j(r)=\min_{1\le i\le m}\PP_i(r)=0.
\]
For this pair, there are $h_0,\ldots,h_{m-1}>0$ with the following
property. For any sequence $S$ vanishing at negative times, set
$X_t=S_t-S_{t-1}$, $D_t=X_t-X_{t-1}$, and define $R_{j,t},W_t$
by~\eqref{eq:necessresiduals}. The sequence
\begin{equation}\label{eq:necessrootrec}
 F_0=S_0,\qquad F_{t+1}=rF_t+R_{j,t},\qquad t\ge0,
\end{equation}
satisfies, with $\beta_r=(r-1)^2/r$,
\begin{equation}\label{eq:necessgrowthidentity}
 D_{t+1}=\beta_r F_t+\frac{r-1}{r}R_{j,t}+\frac1rW_t
                           +\frac1r\sum_{i=0}^{m-1}h_iD_{t-i},
 \qquad t\ge0.
\end{equation}
\end{lemma}
\begin{proof}
By~\eqref{eq:adjacent} and the algebraic bound in the proof of
Lemma~\ref{lem:schedule},
$\min_i\PP_i(b_1)=\PP_1(b_1)\ge b_1>0$.
Failure of~\eqref{eq:criterion} therefore gives $x_0\in(1,b_1)$ with
$\min_i\PP_i(x_0)\le0$. Continuity and $\PP_i(1)=1$ give
$r\in(1,b_1)$ with $\min_i\PP_i(r)=0$.
Choose a minimizing index $j$. Equation~\eqref{eq:adjacent} gives
\begin{equation}\label{eq:necessactivity}
 b_j\ge r\ge b_{j+1}.
\end{equation}
Appendix~\ref{app:growth} derives~\eqref{eq:necessgrowthidentity}
and proves $h_i>0$ for $0\le i\le m-1$.
\end{proof}
The sequence $F_t$ accumulates the residuals with geometric weights.
If $R_{j,t},W_t\ge-\ell_t$ and the past $D_s$ are nonnegative,
\eqref{eq:necessgrowthidentity} gives
$D_{t+1}\ge\beta_r F_t-\ell_t$.

When~\eqref{eq:criterion} fails, choose $j,r$ as in
Lemma~\ref{lem:necessroot}. Fix $1<\rho<r$ and put
$B=2+\sum_{i=1}^m b_i$. Set $y_0=e^A$ with $A\ge1$ sufficiently large
as specified below, independently of the feedback. Then
$X_0=S_0=D_0=F_0=A$ and $D_s=0$ for $s<0$.
For any Borel feedback, work on its event $\mathcal D$.
We prove separation~\eqref{eq:necessseparation} together with
\begin{equation}\label{eq:necessinduction}
 F_t\ge A\rho^t,\qquad S_t\le B^tF_t,\qquad t\ge0.
\end{equation}
The bound on $S_t$ controls $\ell_t/F_t$.
The claims hold at $t=0$. Suppose they hold through $t$.
Lemma~\ref{lem:necesssample} gives $R_{j,t},W_t\ge-\ell_t$ with
\[
 \frac{\ell_t}{F_t}
 \le C\frac{(t+1)^2+t\log B+\log(2+F_t)}{F_t}
 \le C'\frac{1+\log A}{A}.
\]
The last bound uses $F_t\ge A\rho^t$, the decrease of
$\log(2+x)/x$ for $x>0$, and
$\sup_{t\ge0}(t+1)^2\rho^{-t}<\infty$.
Thus $C'$ is independent of $t$, $A$, and the feedback.
Choose $A$ so that $C'(1+\log A)/A\le\min\{r-\rho,\beta_r/2\}$
and $\beta_r A/2\ge\log K$.
Since the past $D_s$ are nonnegative,
\eqref{eq:necessrootrec} and~\eqref{eq:necessgrowthidentity} give
\begin{equation}\label{eq:necessfinalgrowth}
 \begin{aligned}
 F_{t+1}&\ge rF_t-\ell_t\ge\rho F_t\ge A\rho^{t+1},\\
 D_{t+1}&\ge\beta_r F_t-\ell_t\ge\frac{\beta_r}{2}F_t
                    \ge\frac{\beta_r}{2}A\rho^t\ge\log K.
 \end{aligned}
\end{equation}
This preserves separation. The $X_s$ terms in $S_{t+1}-R_{j,t}$
have $s\le t$ and lie between zero and $S_t$.
Thus~\eqref{eq:necessresiduals} and~\eqref{eq:necessrootrec} give
\[
 S_{t+1}\le(B-1)S_t+F_{t+1}\le B^{t+1}F_{t+1},
\]
where the last bound uses $S_t\le B^tF_t$, $F_{t+1}\ge F_t$, and $B^t\ge1$.
This closes the induction. By~\eqref{eq:necessfinalgrowth},
$X_t\to\infty$ and hence $Y_t\to\infty$ on $\mathcal D$.
Consequently, $T^{-1}\sum_{t=0}^{T-1}y_{t+1}^2\to\infty$ as $T\to\infty$.

Let $\mathcal B$ be this escape event and $\mu$ the auxiliary Gaussian
prior. Since $\mathcal D\subset\{\vartheta\in\Theta\}\cap\mathcal B$,
\[
 0<p_0/2\le\P(\mathcal D)
 \le\int_\Theta\P_\vartheta(\mathcal B)\,\mu(\dd\vartheta).
\]
Therefore some fixed $\vartheta\in\Theta$ has
$\P_\vartheta(\mathcal B)>0$. Conditional on this parameter, the noise
has its original iid $N(0,1)$ law. Undoing the deterministic scaling
proves~\eqref{eq:escape} for every prescribed $\sigma>0$.

\section{Conclusion}
For scalar stochastic systems with finitely many unknown coefficients
of power functions and an unknown nonzero input gain, we have established
a polynomial criterion for global stabilizability. The criterion is
necessary and sufficient and requires strict positivity of every
$\PP_j$ on $[1,b_1]$.

When the criterion holds, the constructed feedback makes the time average
of the squared output converge almost surely to the noise variance,
without knowledge of parameter bounds, the gain sign, or the noise
variance. When it fails, every Borel feedback admits a fixed parameter
and a deterministic initial state for which the output magnitude and
the time average of its square diverge with positive probability.
Thus the same criterion separates systems for which the noise variance
is attainable from those for which global stabilization is impossible,
including cases where a criterion polynomial has a multiple zero.
\appendix
\section{Proof of Lemma~\ref{lem:necesssample}}\label{app:centers}
We work in the auxiliary Gaussian model of Section~\ref{sec:necessity}.
The product in~\eqref{eq:envelopeprob} is positive because each factor is
positive and the deficits from one are summable,
$\sum_{s=1}^{\infty}2\{1-\Phi(Ws)\}\le\sum_{s=1}^{\infty}2e^{-W^2s^2/2}<\infty$.

\subsection{Conditional information bounds}
Write $f_t=f(y_t)$ and $\varphi_t=(f_t^\top,u_t)^\top$.
For every Borel feedback, fixing $y_0,\ldots,y_t$ fixes each past input
$u_s=\pi_s(y_0,\ldots,y_s)$, $s<t$. Completing the likelihood square
therefore gives
\begin{equation}\label{eq:necessposterior}
 \vartheta\mid\F_t\sim N(\widehat\vartheta_t,I_t^{-1}),\qquad
 I_t=I_{m+1}+\sum_{s=0}^{t-1}\varphi_s\varphi_s^\top,\qquad
 \widehat\vartheta_t=I_t^{-1}\sum_{s=0}^{t-1}\varphi_s y_{s+1}.
\end{equation}
Define
\begin{equation}\label{eq:necessinfo}
 \Delta_t=\sqrt{\det I_t},\quad
 \sigma_t^2=\Var(y_{t+1}\mid\F_t)=1+\varphi_t^\top I_t^{-1}\varphi_t,
 \quad M_t=\E[y_{t+1}^2\mid\F_t].
\end{equation}
Take $\sigma_t>0$. The matrix determinant lemma gives $\Delta_{t+1}=\sigma_t\Delta_t$, with $\Delta_0=1$.
Introduce the drift and output Gram matrices
\begin{equation}\label{eq:necessgrams}
 B_t=I_m+\sum_{s=0}^{t-1}f_sf_s^\top,\qquad
 G_t=I_{m+1}+\sum_{s=0}^{t-1}\psi_s\psi_s^\top,\qquad
 \psi_s=(f_s^\top,y_{s+1})^\top.
\end{equation}
\begin{lemma}\label{lem:necessprediction}
Put $a_t=(f_t^\top,0)^\top$. For every output history and all $t\ge0$,
\begin{equation}\label{eq:necessprediction}
 \sigma_t^2\ge1+f_t^\top B_t^{-1}f_t,\qquad
 M_t\ge a_t^\top G_t^{-1}a_t.
\end{equation}
There is a constant $C$, depending only on $\Theta,W$, such that on $\Omega$,
\begin{equation}\label{eq:necessdetcompare}
 \det G_t\le C(t+1)^3\det I_t,\qquad t\ge0.
\end{equation}
\end{lemma}
\begin{proof}
For vectors $b_s,x$ of the same dimension and arbitrary weights $\lambda_s$,
the Cauchy--Schwarz inequality gives
\[
 x^\top\left(I+\sum_s b_sb_s^\top\right)^{-1}x
 \le\norm{x-\sum_s\lambda_sb_s}^2+\sum_s\lambda_s^2.
\]
Put $q=I_t^{-1}\varphi_t$, $\lambda_s=\varphi_s^\top q$ for $0\le s<t$,
and $\mu_t=\E[y_{t+1}\mid\F_t]$.
By~\eqref{eq:necessposterior} and the definition of $I_t$,
\[
 \varphi_t-\sum_{s<t}\lambda_s\varphi_s=q,\qquad
 \mu_t=\sum_{s<t}\lambda_sy_{s+1},\qquad
 \norm q^2+\sum_{s<t}\lambda_s^2=q^\top I_tq=\sigma_t^2-1.
\]
Let $q_f$ contain the first $m$ coordinates of $q$.
Applying the Cauchy--Schwarz inequality with $b_s=f_s$ gives
\[
 f_t^\top B_t^{-1}f_t
 \le\norm{q_f}^2+\sum_{s<t}\lambda_s^2\le \sigma_t^2-1.
\]
Since $a_t-\sum_{s<t}\lambda_s\psi_s=(q_f^\top,-\mu_t)^\top$,
the same bound with $b_s=\psi_s$ gives
\[
 a_t^\top G_t^{-1}a_t
 \le\norm{q_f}^2+\mu_t^2+\sum_{s<t}\lambda_s^2\le M_t-1.
\]

For the determinant comparison, put
$T=\left(\begin{smallmatrix}I_m&0\\\theta^\top&g\end{smallmatrix}\right)$
and let $e$ be the last standard basis vector of $\R^{m+1}$.
The system equation gives $T^{-1}\psi_s=\varphi_s+g^{-1}w_{s+1}e$.
On $\Theta$, both $T^{-1}$ and $g^{-1}$ are uniformly bounded. Hence,
for some $C_*>0$ depending only on $\Theta$,
\[
 T^{-1}G_tT^{-\top}
 \preceq C_*\left(I_t+\sum_{s=1}^t w_s^2\,ee^\top\right).
\]
Taking determinants and using $I_t\succeq I_{m+1}$ and $\det T=g$ give
\[
 \det G_t\le C_*^{m+1}g^2\det I_t
       \left(1+\sum_{s=1}^t w_s^2\,e^\top I_t^{-1}e\right)
 \le C(t+1)^3\det I_t
 \qquad\text{on }\Omega,
\]
where $C$ depends only on $\Theta,W$. This proves~\eqref{eq:necessdetcompare}.
\end{proof}

For $1\le k\le t\wedge m$, define
\begin{equation}\label{eq:necessminor}
 \mathscr D_{t,k}=\det\begin{pmatrix}
 f_1(y_t)&\cdots&f_k(y_t)&0\\
 f_1(y_{t-1})&\cdots&f_k(y_{t-1})&y_t\\
 \vdots&&\vdots&\vdots\\
 f_1(y_{t-k})&\cdots&f_k(y_{t-k})&y_{t-k+1}
 \end{pmatrix}.
\end{equation}
For $\lambda>0$, form a matrix with the rows of $I_{m+1}$,
the rows $\psi_j^\top$ for $0\le j<t$, and the row
$\sqrt\lambda\,a_t^\top$. Its Gram matrix is $G_t+\lambda a_ta_t^\top$.
In the Cauchy--Binet formula, select $\sqrt\lambda\,a_t^\top$,
$\psi_{t-k}^\top,\ldots,\psi_{t-1}^\top$, and the unit rows in
coordinates $k+1,\ldots,m$. Expanding along the unit rows gives
$\lambda|\mathscr D_{t,k}|^2$ as one term. All other terms are
nonnegative, so the matrix determinant lemma gives
\[
 \det G_t\bigl(1+\lambda a_t^\top G_t^{-1}a_t\bigr)
 =\det(G_t+\lambda a_ta_t^\top)
 \ge\lambda|\mathscr D_{t,k}|^2.
\]
Dividing by $\lambda$ and letting $\lambda\to\infty$, then applying
Lemma~\ref{lem:necessprediction}, yields, on $\Omega$,
\begin{equation}\label{eq:necessbridge}
 M_t\ge a_t^\top G_t^{-1}a_t
 \ge\frac{|\mathscr D_{t,k}|^2}{\det G_t}
 \ge\frac{|\mathscr D_{t,k}|^2}{C(t+1)^3\Delta_t^2}.
\end{equation}

\subsection{Determinant estimates}\label{app:determinants}
We control cancellation in the leading terms of $\mathscr D_{t,k}$.
\begin{lemma}\label{lem:necessgeometry}
For each $2\le k\le m$, there exist $K>1$ and $\rho_k>0$, depending
only on $k,b_1,\ldots,b_k$, with the following property.
Let $s_i\in\{-1,1\}$, $0\le i\le k$, and
$z_0\ge Kz_1\ge\cdots\ge K^kz_k>0$.
Write $\mathscr D=\left.\mathscr D_{t,k}\right|_{y_{t-i}=s_i z_i,\;0\le i\le k}$
with $\mathscr D_{t,k}$ defined in~\eqref{eq:necessminor}, and set
\begin{equation}\label{eq:necessmonomials}
 M_a^*=z_{a-1}\prod_{i=1}^a z_{i-1}^{b_i}
                    \prod_{i=a+1}^k z_i^{b_i},\qquad
 \mathscr M=\sum_{a=1}^kM_a^*.
\end{equation}
With $\partial_{\log z_i}=z_i\partial_{z_i}$, we have
\begin{equation}\label{eq:necessnondegenerate}
 \left|\frac{\mathscr D}{\mathscr M}\right|
 +\sum_{i=0}^{k-2}\left|\partial_{\log z_i}
 \frac{\mathscr D}{\mathscr M}\right|\ge\rho_k.
\end{equation}
\end{lemma}
\begin{proof}
We control cancellation of the leading sum through its derivatives.
Expand $\mathscr D$ along the last column in~\eqref{eq:necessminor}.
For row $a+1$, pairing $z_0,\ldots,z_k$ except $z_a$ with
$b_1,\ldots,b_k$ in decreasing order, then multiplying by $z_{a-1}$,
gives $M_a^*$. Any other permutation has an inverted pair. Correcting
it increases the magnitude by $(z_p/z_q)^{b_i-b_j}\ge K^\eta$,
where $p<q$, $i<j$, and $\eta=\min_{i<k}(b_i-b_{i+1})$.
Each of the $k!-1$ other products is therefore at most
$K^{-\eta}M_a^*$. Summing gives
\[
 \mathscr D=\sum_{a=1}^k\epsilon_a M_a^*+E,
 \qquad \epsilon_a\in\{-1,1\},\qquad
 |E|\le(k!-1)K^{-\eta}\mathscr M.
\]
Here $E$ collects the remaining terms, and the signs $\epsilon_a$ are
fixed as $z_i$ vary. Since every variable exponent lies in $[0,b_1+1]$,
\[
 \sum_{i=0}^{k-2}|\partial_{\log z_i}E|
 \le(k-1)(b_1+1)(k!-1)K^{-\eta}\mathscr M.
\]
Put $\omega_a=\epsilon_aM_a^*$ for $1\le a\le k$ and $P=\sum_{a=1}^k\omega_a$.
The different exponents let us recover each $\omega_a$ from $P$ and
its derivatives. Direct differentiation gives
\[
 \begin{aligned}
 \omega_1&=(\partial_{\log z_0}-b_1)P,\\
 \omega_{i+1}&=(\partial_{\log z_i}-b_{i+1})P+b_{i+1}\omega_i
       -(b_i-b_{i+1})\sum_{a=1}^{i-1}\omega_a,\quad1\le i\le k-2.
 \end{aligned}
\]
Using these relations successively and $\omega_k=P-\sum_{a=1}^{k-1}\omega_a$ gives
\[
 \mathscr M=\sum_{a=1}^k|\omega_a|
 \le C_{k,b}\left(|P|+\sum_{i=0}^{k-2}|\partial_{\log z_i}P|\right),
\]
where $C_{k,b}\ge1$ depends only on $k,b_1,\ldots,b_k$.
For sufficiently large $K$, the two error bounds therefore give
$|\mathscr D|+\sum_{i=0}^{k-2}|\partial_{\log z_i}\mathscr D|
\ge\mathscr M/(2C_{k,b})$.
Finally, $0\le\partial_{\log z_i}\log\mathscr M\le b_1+1$, so the identity
\[
 \frac{\partial_{\log z_i}\mathscr D}{\mathscr M}
 =\partial_{\log z_i}\!\left(\frac{\mathscr D}{\mathscr M}\right)
  +\frac{\mathscr D}{\mathscr M}\,
       \partial_{\log z_i}\log\mathscr M
\]
gives~\eqref{eq:necessnondegenerate} by summing absolute values.
\end{proof}

Choose one $K$ for all $2\le k\le m$. The same exchange estimate
allows us to enlarge $K$ so that, for every $1\le r\le m$ and
$1\le i_1<\cdots<i_r\le m$, outputs ordered by magnitude with
$|y_{s_\ell}|\ge K|y_{s_{\ell+1}}|$ for $1\le\ell<r$ and $|y_{s_r}|>0$
satisfy
\[
 \left|\det\bigl(f_{i_j}(y_{s_\ell})\bigr)_{\ell,j=1}^r\right|
 \ge\frac12\prod_{\ell=1}^r|y_{s_\ell}|^{b_{i_\ell}}.
\]
One $K$ suffices for the finitely many exponent subsets and permutations.
For $m=1$, take any $K>1$.

\subsection{A noncancellation event}
The next lemma converts a lower bound on a function and its derivatives
into a lower bound in terms of distances from finite sets.

\begin{lemma}\label{lem:centers}
Let $\mathcal G_\lambda(x)$ be one fixed finite expression in real variables
$(\lambda,x)$ using arithmetic operations, $\exp$, and $\log$.
For each $\lambda\in\Lambda$, assume it is real-valued and $C^1$ near
$[-1,1]^q$, with
\begin{equation}\label{eq:centerhyp}
 |\mathcal G_\lambda(x)|+\sum_{i=1}^q|\partial_i \mathcal G_\lambda(x)|\ge\rho>0
 \qquad \lambda\in\Lambda,\quad x\in[-1,1]^q.
\end{equation}
There exist an integer $N\ge1$ and Borel measurable functions
$c_{i,j}(\lambda,x_1,\ldots,x_{i-1})\in[-1,1]$,
$1\le i\le q$, $1\le j\le N$, such that, writing
$\mathcal C_i=\{c_{i,1},\ldots,c_{i,N}\}$ with repetitions allowed,
\begin{equation}\label{eq:centerineq}
 |\mathcal G_\lambda(x)|\ge\frac{\rho}{4^q q!}
       \prod_{i=1}^q\operatorname{dist}
           (x_i,\mathcal C_i(\lambda,x_1,\ldots,x_{i-1})).
\end{equation}
\end{lemma}
\begin{proof}
For $0\le i\le q$ define
\[
 Q_i(\lambda,x_1,\ldots,x_i)
 =\min_{x_{i+1},\ldots,x_q\in[-1,1]}
   \left\{|\mathcal G_\lambda(x)|+\sum_{j=i+1}^q|\partial_j \mathcal G_\lambda(x)|\right\}.
\]
The minima are attained by compactness, and $Q_0\ge\rho$,
$Q_q=|\mathcal G_\lambda|$. It suffices to prove
\[
 Q_i(\lambda,x_1,\ldots,x_i)
 \ge\frac{Q_{i-1}(\lambda,x_1,\ldots,x_{i-1})}{4(q-i+1)}
           \operatorname{dist}(x_i,\mathcal C_i).
\]

On the full set of parameters satisfying the hypotheses, the functions
$Q_i$ and their sets of minimizers are definable in the real exponential
structure, which is o-minimal~\cite[Theorem~5.2]{Wilkie2009}.
Definable choice~\cite[Theorem~3.1]{Coste1999} gives a minimizer
$\zeta(\lambda,x_1,\ldots,x_i)\in[-1,1]^{q-i}$ for $Q_i$.

For each $i$, write $u=x_i$ and abbreviate
$a=Q_{i-1}(\lambda,x_1,\ldots,x_{i-1})$.
The monotonicity and decomposition theorems
\cite[Theorems~2.1, 6.6, and 6.7]{Coste1999} give a definable family
of finite partitions of $[-1,1]$ into points and intervals. On each
interval, $Q_i$ and each coordinate of $\zeta$ are $C^1$ and monotone
in $u$, and $Q_i-a/2$ has constant sign.
The partition endpoints form finite definable sets depending on
$(\lambda,x_1,\ldots,x_{i-1})$.
Uniform finiteness~\cite[Theorem~2.9]{Coste1999} bounds their cardinalities
by one integer $N$ for all parameters and $i$.
Take these endpoints as $\mathcal C_i$ when $a>0$, and take
$\mathcal C_i=\{-1,1\}$ when $a=0$.
Ordering the endpoints and repeating entries gives definable, hence Borel,
functions $c_{i,1},\ldots,c_{i,N}$.

Fix $\lambda,x_1,\ldots,x_{i-1}$ and write $Q_i(u)$ and $\zeta(u)$.
If $a=0$, $u\in\mathcal C_i$, or $Q_i(u)\ge a/2$, the desired bound
follows from $\operatorname{dist}(u,\mathcal C_i)\le2$.
Otherwise, on the interval containing $u$,
$|\partial_i\mathcal G_\lambda(s,\zeta(s))|\ge a-Q_i(s)>a/2$,
so this derivative has constant sign. Choose the endpoint $\xi$
in the direction in which $Q_i$ does not increase.
For $v$ between $u$ and $\xi$, both $|\mathcal G_\lambda|$ and each
$|\partial_j\mathcal G_\lambda|$, $j>i$, along the segment are bounded
by $Q_i(u)$. Each coordinate of $\zeta$ has total variation at most two.
Thus the chain rule gives
\begin{align*}
 \frac a2|u-v|
 &\le |\mathcal G_\lambda(u,\zeta(u))-\mathcal G_\lambda(v,\zeta(v))|\\
 &\quad+\sum_{j=i+1}^q\int_{\min(u,v)}^{\max(u,v)}
       |\partial_j\mathcal G_\lambda(s,\zeta(s))|\,|\zeta'_j(s)|\,\dd s
 \le2(q-i+1)Q_i(u).
\end{align*}
Since $\xi\in\mathcal C_i$, letting $v\to\xi$ gives
$Q_i(u)\ge\frac{a|u-\xi|}{4(q-i+1)}
\ge\frac{a}{4(q-i+1)}\operatorname{dist}(u,\mathcal C_i)$.
Iterating over $i=1,\ldots,q$, with $Q_0\ge\rho$ and
$Q_q=|\mathcal G_\lambda|$, proves~\eqref{eq:centerineq} on $\Lambda$.
\end{proof}

\begin{lemma}\label{lem:necessgood}
There are constants $0<\delta\le1/2$ and $c>0$, depending only on
$m,b,\Theta,W$, such that every finite Borel feedback and deterministic
initial state admits an event $\mathcal D\subset\Omega$ with
$\P(\mathcal D)\ge p_0/2$. On this event, both estimates below hold simultaneously.
For all $s\ge1$,
\begin{equation}\label{eq:necesszeroevent}
 Y_s\ge\delta s^{-2}\sqrt{M_{s-1}}.
\end{equation}
For every $t\ge1$ at which the outputs satisfy~\eqref{eq:necessseparation},
\begin{equation}\label{eq:necessminorlower}
 |\mathscr D_{t,t\wedge m}|\ge
 \frac{c\mathscr M}{(t+1)^{2[(t\wedge m)-1]}(2+\log Y_t)^{3[(t\wedge m)-1]}}
 \prod_{s=t-(t\wedge m)+2}^t\frac{\sigma_{s-1}}{Y_s},
\end{equation}
where $\mathscr M$ is given by~\eqref{eq:necessmonomials} with
$k=t\wedge m$ and $z_j=Y_{t-j}$.
For $t\wedge m=1$, $\mathscr M=Y_t^{b_1+1}$ and the product is one.
\end{lemma}
\begin{proof}
For each $t$ with $k=t\wedge m\ge2$, put $q=k-1$ and $s_i=t-q+i$,
$1\le i\le q$. Fix the two oldest magnitudes
$Y_{t-k+1}\ge KY_{t-k}>0$ and the output signs.
For real $n\ge1$ and $x\in[-1,1]^q$, set
\[
 Y_{s_i}=KY_{s_i-1}e^{n(1+x_i)},\qquad
 F(x)=\frac{\mathscr D}{\mathscr M},\qquad z_j=Y_{t-j}.
\]
The parameters are $n$ and the two oldest magnitudes, with finitely many
choices of $k$ and the signs. Since $z^b=\exp(b\log z)$ and
$\mathscr M>0$, this is a smooth family of the form required by
Lemma~\ref{lem:centers}. The chain rule gives
\[
 \begin{aligned}
 n\partial_{\log Y_{s_i}}F
 &=\partial_{x_i}F-\partial_{x_{i+1}}F,\qquad 1\le i<q,\\
 n\partial_{\log Y_{s_q}}F&=\partial_{x_q}F.
 \end{aligned}
\]
Summing absolute values and using $n\ge1$ in~\eqref{eq:necessnondegenerate}
gives $|F|+\sum_i|\partial_{x_i}F|\ge\rho_k/2$.
Lemma~\ref{lem:centers} gives sets $\mathcal C_i$ satisfying
\eqref{eq:centerineq} with $\rho=\rho_k/2$ and cardinalities bounded
uniformly over this parameter family.
For integer $n\ge1$, let $\Xi_{t,n,i}$ collect the centers
\[
 \xi=\epsilon_iKY_{s_i-1}e^{n(1+\xi_i)},\qquad
 \xi_i\in\mathcal C_i,\quad \epsilon_i\in\{-1,1\},
\]
over all sign patterns. Set $\Xi_{t,n,i}=\varnothing$ if the oldest pair
is not positive and separated or the preceding coordinates are undefined
or outside the box. These sets satisfy $|\Xi_{t,n,i}|\le N_{m,b}$ uniformly
in $t,n,i$ and the output history. Their centers are
$\F_{s_i-1}$-measurable because $\mathcal C_i$ depends only on the oldest
magnitudes and the preceding coordinates.

For a normal variable $N$ of mean $\mu$ and variance $v^2>0$, the following
Gaussian estimates hold with an absolute constant $C_0$, for $\xi\in\R$
and $0<\varepsilon\le1/2$,
\begin{equation}\label{eq:smallballs}
 \P\{|N-\xi|<\varepsilon v\}\le C_0\varepsilon,\qquad
 \P\{|N|<\varepsilon\sqrt{\mu^2+v^2}\}\le C_0\varepsilon.
\end{equation}
The first bound follows from the maximum density. For the second,
standardize by $v$. When $|\mu|/v\le2$, the interval length is at
most $2\sqrt5\varepsilon$. Otherwise it is at most
$\sqrt5\varepsilon|\mu|/v$ and its distance from the standardized
mean is at least $2|\mu|/(5v)$. Density times length is then bounded by
$\sqrt{5/(2\pi)}\varepsilon(|\mu|/v)e^{-2\mu^2/(25v^2)}\le C_0\varepsilon$.

For $0<\delta\le1/2$, let $\mathcal E$ require~\eqref{eq:necesszeroevent}
for all $s\ge1$ and, for all $t,n,i$ above,
\begin{equation}\label{eq:necesscenterevent}
 |y_{s_i}-\xi|\ge\frac{\delta\sigma_{s_i-1}}{4(t+1)^2n^2}
 \qquad \xi\in\Xi_{t,n,i}.
\end{equation}
Apply~\eqref{eq:smallballs} conditionally and sum over at most
$(m-1)N_{m,b}$ centers for each $t,n$. The convergence of
$\sum_{s\ge1}s^{-2}$ and $\sum_{t,n\ge1}(t+1)^{-2}n^{-2}$ gives
$\P(\mathcal E^c)\le C_{m,b}\delta$.
Choose $\delta\le\min\{1/2,p_0/(2C_{m,b})\}$ and set
$\mathcal D=\Omega\cap\mathcal E$. Then
$\P(\mathcal D)\ge p_0-C_{m,b}\delta\ge p_0/2$.

On $\mathcal D$, assume~\eqref{eq:necessseparation} and set $k=t\wedge m$.
If $k=1$, the identity $|\mathscr D_{t,1}|=\mathscr M=Y_t^{b_1+1}$
proves~\eqref{eq:necessminorlower} with $c\le1$, including $m=1$.
For $k\ge2$, the event $\mathcal E$ covers all integers $n\ge1$,
so choose $n=\lceil1+\log Y_t\rceil$.
Then $\log(Y_{s_i}/Y_{s_i-1})-\log K\in[0,\log Y_t]$,
so $x_i\in[-1,1]$ and $n\le2+\log Y_t$.

Write $\varepsilon_i=\delta\sigma_{s_i-1}/[4(t+1)^2n^2]$.
Since $s_i\le t$ and $\sigma_{s_i-1}\le\sqrt{M_{s_i-1}}$,
\eqref{eq:necesszeroevent} gives
$\varepsilon_i\le Y_{s_i}/4$.
For each $\xi_i\in\mathcal C_i$, its output center with the actual sign
satisfies $|Y_{s_i}-|\xi||\ge\varepsilon_i$ by~\eqref{eq:necesscenterevent}.
Monotonicity of $\log$ gives
\[
 \operatorname{dist}(x_i,\mathcal C_i)
 \ge\frac1n\log\left(1+\frac{\varepsilon_i}{Y_{s_i}}\right)
 \ge\frac{\delta\sigma_{s_i-1}}{8(t+1)^2n^3Y_{s_i}}.
\]
Thus~\eqref{eq:centerineq}, with $q=k-1$, yields
\[
 |\mathscr D_{t,k}|\ge
 \frac{c\mathscr M}{(t+1)^{2q}n^{3q}}
 \prod_{s=t-q+1}^t\frac{\sigma_{s-1}}{Y_s}.
\]
The constant $c>0$ is uniform over $k\le m$.
Using $n\le2+\log Y_t$ proves~\eqref{eq:necessminorlower}.
\end{proof}

\begin{proof}[Proof of Lemma~\ref{lem:necesssample}]
In this proof, $c,C>0$ depend only on $m,b,\Theta,W$ and may change
between estimates.
Choose $K$ as in Section~\ref{app:determinants} and take the event
$\mathcal D$ supplied by Lemma~\ref{lem:necessgood}.
Since $M_{s-1}\ge \sigma_{s-1}^2\ge1$, \eqref{eq:necesszeroevent} ensures
$Y_s>0$ for every $s\ge1$.
Fix $t\ge0$ for which~\eqref{eq:necessseparation} holds.

Since $I_0=I_{m+1}$, we have $M_0\ge Y_0^{2b_1}$ and hence
$X_1\ge b_1X_0+\log\delta$ by~\eqref{eq:necesszeroevent}.
It follows directly from~\eqref{eq:necessresiduals} that
$R_{j,0},W_0\ge\log\delta$, which covers $t=0$ after increasing $C$.

For $t\ge1$, take $k=t\wedge m$ and $z_i=Y_{t-i}$
in~\eqref{eq:necessmonomials}. For $1\le a\le t\wedge m$, set
\[
 L_a(t)=\log M_a^*=\sum_{i=1}^a b_iX_{t-i+1}+X_{t-a+1}
                   +\sum_{i=a+1}^{t\wedge m}b_iX_{t-i}.
\]
Combining~\eqref{eq:necessbridge} with~\eqref{eq:necesszeroevent} yields
\[
 Y_{t+1}\ge c(t+1)^{-4}|\mathscr D_{t,t\wedge m}|/\Delta_t.
\]
Put $n=t-k+1$. Since $\Delta_s=\sigma_{s-1}\Delta_{s-1}$ and
$\sigma_{s-1}\le\delta^{-1}s^2Y_s$ by~\eqref{eq:necesszeroevent},
\begin{equation}\label{eq:necesslogH}
 \frac1{\Delta_t}\prod_{s=n+1}^t\frac{\sigma_{s-1}}{Y_s}
 =\frac1{\Delta_n\prod_{s=n+1}^tY_s}
 \ge\frac{\delta^n}{(n!)^2\prod_{s=1}^tY_s}.
\end{equation}
Insert~\eqref{eq:necessminorlower} and~\eqref{eq:necesslogH} into the bound
on $Y_{t+1}$. Taking logarithms and using
$\mathscr M=\sum_{a=1}^k e^{L_a(t)}$ and $n\le t$ yields
\begin{equation}\label{eq:necesscumulative}
 \sum_{s=1}^{t+1}X_s
 \ge\max_{1\le a\le t\wedge m}L_a(t)-C(t+1)\log(t+2)
                         -3(m-1)\log(2+X_t).
\end{equation}
For $j\le t$, we have $S_{t+1}-R_{j,t}=L_j(t)$.
For $j>t$, the zero values at negative times give
\[
 S_{t+1}-R_{j,t}-L_t(t)
 \le(b_{t+1}+1)X_0-X_1\le X_0-\log\delta.
\]
Since $S_{t+1}=X_0+\sum_{s=1}^{t+1}X_s$ and $0\le X_t\le S_t$
under~\eqref{eq:necessseparation},
\eqref{eq:necesscumulative} proves $R_{j,t}\ge-\ell_t$ after increasing $C$.

To bound $W_t$, put
\[
 U_t=\sum_{i=1}^{t\wedge m}b_iX_{t-i},\qquad t\ge0.
\]
With $\mathsf F=(I_m,f_0,\ldots,f_{t-1})^\top$, the drift Gram matrix
in~\eqref{eq:necessgrams} satisfies $B_t=\mathsf F^\top\mathsf F$.
Apply the Cauchy--Binet formula to this factorization.
Each selection of $m$ rows of $\mathsf F$ contains at most $t\wedge m$ data
rows. Expanding along the selected identity rows leaves a minor of
$(f_i(y_s))_{0\le s<t,\,1\le i\le m}$.
Rearrangement bounds every permutation term
by $e^{U_t}$. Terms with fewer rows have the same upper bound because
all $X_s\ge0$ and all exponents are positive. There are at most
$C_m(t+1)^m$ minors, so
\[
 \det B_t\le C_m(t+1)^m e^{2U_t}.
\]
For $r=(t+1)\wedge m$, the expansion of $\det B_{t+1}$ contains the
squared minor formed by $f_{t-r+1}^\top,\ldots,f_t^\top$ and the unit rows
in coordinates $r+1,\ldots,m$. The determinant bound in
Section~\ref{app:determinants} gives $\det B_{t+1}\ge\tfrac14e^{2U_{t+1}}$.
Since $B_{t+1}=B_t+f_tf_t^\top$, Lemma~\ref{lem:necessprediction} and
the matrix determinant lemma give
\[
 \sigma_t^2\ge1+f_t^\top B_t^{-1}f_t
 =\frac{\det B_{t+1}}{\det B_t}
 \ge c(t+1)^{-m}e^{2(U_{t+1}-U_t)}.
\]
Combine this with $Y_{t+1}\ge\delta(t+1)^{-2}\sigma_t$
from~\eqref{eq:necesszeroevent}. Taking logarithms and using
$U_{t+1}-U_t=\sum_{i=1}^m b_iD_{t-i+1}$ give
\[
 X_{t+1}\ge\sum_{i=1}^m b_iD_{t-i+1}-C\log(t+2),
\]
including the times $t<m$. Taking $C$ sufficiently large gives both
$R_{j,t}\ge-\ell_t$ and $W_t\ge-\ell_t$.
\end{proof}

\Needspace{6\baselineskip}
\section{Polynomial identity and coefficient positivity}\label{app:growth}
\begin{proof}[Completion of the proof of Lemma~\ref{lem:necessroot}]
For the pair $j,r$ selected in Section~\ref{sec:necessity}, define
\[
 A_j(z)=\sum_{i=1}^j b_i z^{i-1}+z^{j-1}
                          +\sum_{i=j+1}^m b_i z^i,\qquad
 V(z)=\sum_{i=1}^m b_i z^{i-1}.
\]
The root equation gives $A_j(1/r)=r^2/(r-1)$. Thus
\begin{equation}\label{eq:necessquotients}
 q_j(z)=\frac{1-z(1-z)A_j(z)}{1-rz},\qquad
 \mathcal H(z)=V(z)+\frac{(r-1)A_j(z)-r^2}{1-rz}
\end{equation}
are polynomials, with $\deg\mathcal H\le m-1$.
Write $\mathcal H(z)=\sum_{i=0}^{m-1}h_i z^i$ and set $h_m=0$.
We compare coefficients in
$(1-rz)\mathcal H=(1-rz)V+(r-1)A_j-r^2$ from both ends.
Starting at the highest degree gives
\[
 \begin{aligned}
 h_{k-1}&=\frac{h_k+b_k-b_{k+1}}r>0,\qquad k=m,m-1,\ldots,j+1,\\
 h_{j-1}&=\frac{h_j}r+b_j-\frac{b_{j+1}}r>0.
 \end{aligned}
\]
The last inequality follows from $r>1$ and~\eqref{eq:necessactivity}.
If $j\ge2$, the constant coefficient gives $h_0=r(b_1-r)>0$.
Since $b_i>r$ for $i<j$, comparison of the remaining lower coefficients
gives, by induction,
\[
 h_i=rh_{i-1}-r(b_i-b_{i+1})
       \ge r(b_{i+1}-r)>0,\qquad 1\le i\le j-2.
\]
Thus every coefficient $h_0,\ldots,h_{m-1}$ is positive.

Using~\eqref{eq:necessquotients} and factoring a difference of squares gives
\begin{equation}\label{eq:necessidentity}
 \begin{aligned}
 (r-1)^2q_j+(1-z)^2(\mathcal H-V)
 &=(r-1)(1-z)A_j+\frac{(r-1)^2-r^2(1-z)^2}{1-rz}\\
 &=(r-1)(1-z)A_j+1-2r+rz.
 \end{aligned}
\end{equation}

Let $\mathsf L$ denote the backward shift, so that
$(\mathsf L S)_t=S_{t-1}$. By~\eqref{eq:necessresiduals} and~\eqref{eq:necessquotients},
$R_{j,t}=(1-r\mathsf L)q_j(\mathsf L)S_{t+1}$.
Since $q_j(0)=1$, the sequence
$q_j(\mathsf L)S_t$ has initial value $S_0$ and satisfies
\eqref{eq:necessrootrec}. It therefore equals $F_t$ for $t\ge0$.
Also $W_t=(1-\mathsf L)S_{t+1}-V(\mathsf L)(1-\mathsf L)^2S_t$.
Applying~\eqref{eq:necessidentity} to $S_t$ and substituting these
expressions gives the recurrence for $D_{t+1}$ in
\eqref{eq:necessgrowthidentity}.
\end{proof}

\Needspace{6\baselineskip}
\section{Proof of Remark~\ref{rem:realexponents}}\label{app:abspowers}
\begin{proof}
Set $f_i(y)=|y|^{b_i}$. For sufficiency, $f_i(y)=y^{b_i}$ at the positive
interpolation nodes in~\eqref{eq:bands} and on $J=[0,1]$
in~\eqref{eq:localgram}. Thus the interpolation estimates and the
matrix $\Gamma_J$ are unchanged.
The initial and final calibration bounds use only $|f_i(y)|=|y|^{b_i}$,
while the remaining estimates in Sections~\ref{sec:supervisor}--\ref{sec:local}
use the regression identity~\eqref{eq:ratioerror} and growth
bound~\eqref{eq:upper}. These properties are preserved, so the same
feedback proves sufficiency.

For necessity, fix signs in $y_{t-i}=s_i z_i$ with $z_i>0$.
Then $f_j(s_i z_i)=z_i^{b_j}$, so each determinant term
in~\eqref{eq:necessminor} keeps its magnitude and exponents, while
its sign may change. The derivative recurrence and remainder bounds in
Lemma~\ref{lem:necessgeometry} hold for all signs, giving
\eqref{eq:necessnondegenerate}.
Since $z^b=\exp(b\log z)$ and $\mathscr M>0$,
$\mathscr D/\mathscr M$ is smooth and uses only arithmetic, $\exp$,
and $\log$, as required by Lemma~\ref{lem:centers}.
The conditional Gaussian estimates hold for any Borel regressor,
so the event construction in Lemma~\ref{lem:necessgood} also applies.
For $k=1$, use $|\mathscr D_{t,1}|=Y_t^{b_1+1}$.
The Gram determinant bounds use only magnitudes and exponent ordering.
Thus Lemma~\ref{lem:necesssample} gives the output bounds, and
Section~\ref{sec:necessity} proves necessity.
Finally, $f_i(\sigma y)=\sigma^{b_i}f_i(y)$ preserves the noise
normalization for every $\sigma>0$.
\end{proof}
\Needspace{6\baselineskip}

\end{document}